\documentclass[a4paper, 11pt]{article}
\title{Equivariant weight filtration for real algebraic varieties with action}
\author{Fabien Priziac}
\date{}
\usepackage{amsfonts}
\usepackage{amsmath}
\usepackage{amsthm}
\usepackage{amssymb}
\usepackage[all]{xy}
\usepackage{graphicx}

\newtheorem{de}{Definition}[section]
\newtheorem{theo}[de]{Theorem}
\newtheorem{prop}[de]{Proposition}
\newtheorem{cor}[de]{Corollary}
\newtheorem{defprop}[de]{Definition and Proposition}

\newtheorem{lem}[de]{Lemma}

\newcommand{\Ext}{\text{Ext}}
\newcommand{\Hom}{\text{Hom}}

\theoremstyle{remark}
\newtheorem{rem}[de]{Remark}
\newtheorem{ex}[de]{Example}

\begin{document}

\maketitle

\begin{abstract}
We show the existence of a weight filtration on the equivariant homology of real algebraic varieties equipped with a finite group action, by applying group homology to the weight complex of McCrory and Parusi\'nski. If the group is of even order, we can not extract additive invariants directly from the induced spectral sequence. 

Nevertheless, we construct finite additive invariants in terms of bounded long exact sequences, recovering Fichou's equivariant virtual Betti numbers in some cases. In the case of the two-elements group, we recover these additive invariants by using globally invariant chains and the equivariant version of Guill\'en and Navarro Aznar's extension criterion. 
\end{abstract}
\footnote{

Research resulting from a doctoral thesis carried out at the laboratory IRMAR (Institute of Mathematical Research of Rennes), inside the University of Rennes 1.}

\section{Introduction}

In \cite{Del}, P. Deligne established the existence of the weight filtration on the rational cohomology of complex algebraic varieties. A real analog of this filtration was introduced by B. Totaro on the Borel-Moore $\mathbb{Z}_2$-homology of real algebraic varieties in \cite{Tot}. Using an extension criterion for functors defined on nonsingular varieties by F. Guill\'en and V. Navarro Aznar (\cite{GNA}), C. McCrory and A. Parusi\'nski showed in \cite{MCP} that this weight filtration for real algebraic varieties can be induced from a functorial (with respect to proper regular morphisms) filtered chain complex, defined up to filtered quasi-isomorphism. Considering the associated spectral sequence (that does not collapse at level two, contrary to the complex framework), they highlight the information contained in this invariant. In particular, McCrory and Parusi\'nski extract from the real weight spectral sequence their additive virtual Betti numbers (\cite{MCP-VB}). Furthermore, they realize the weight complex from the chain level, using resolution of singularities (\cite{Hiro}). On the real algebraic varieties, this geometric filtration coincides with a filtration defined on the category of $\mathcal{AS}$-sets (\cite{Kur}, \cite{KP}) and continuous proper maps with $\mathcal{AS}$-graph, using Nash-constructible functions (\cite{MCP-ACF}, \cite{MCP}).
   
In \cite{Pri-CA}, we considered real algebraic varieties equipped with a finite group action. By taking advantage of the functoriality of the weight complex, we equipped it with the induced group action. We then used an equivariant version of the extension criterion of Guill\'en and Navarro Aznar (theorem \ref{crit_action}) to show the uniqueness of the weight complex with action, with respect to extension, acyclicity and additivity properties, up to equivariant filtered quasi-isomorphism (theorem \ref{th_weight-complex-act}). Focusing on its realization by the Nash-constructible filtration, we obtained a filtered version of the Smith short exact sequence for an algebraic involution (theorem \ref{nash-smith_ex}). Futhermore, in the case of an involution acting without fixed point on a compact real algebraic variety, we related by a filtered isomorphism the invariant filtered chains and the filtered chains of the arc-symmetric quotient (proposition \ref{fil_Nash_quotient}).

In this paper, we apply a functor (definition and proposition \ref{fonct_L}) to the weight complex with action (definition \ref{def_eq_we_co}) in order to induce a weight filtration on the equivariant homology of real algebraic varieties equipped with a finite group action. We show the uniqueness of the equivariant weight complex, with respect to extension, acyclicity and additivity properties, similarly to the previous frameworks (theorem \ref{equiv_wei_comp}). Nevertheless, some significative differences appear between the induced spectral sequence and the weight spectral sequence. In particular, the equivariant weight spectral sequence is no longer left bounded (corollaries \ref{equi_weight_spec_max_dim} and \ref{equiv_weight_spec_bounds}). As a consequence, the long exact sequences of additivity provided by the equivariant weight spectral sequence are not finite in general. Moreover, the equivariant weight spectral sequence of a compact nonsingular variety does not degenerate at level two in general (proposition \ref{var_comp-non-sing_equiv}). This prevents us to recover some additive invariants directly from the page two of the equivariant weight spectral sequence, unlike the non-equivariant set-up.

In the final part of this work, we identify finite additive invariants in terms of bounded long exact sequences of spectral sequences, provided by bounded double complexes which extend the additivity of the Nash-constructible filtration and of the weight spectral sequence, and are related to the equivariant weight spectral sequence. In the case of the two-elements group, we use the Smith Nash-constructible exact sequence to identify additive invariants involving the geometry of invariant chains and the geometry of the fixed points set. We then show that they coincide with G. Fichou's additive equivariant virtual Betti numbers (\cite{GF}) in some cases (theorem \ref{coinc_b-k-g_eq-virt-bet}). After having recovered the equivariant virtual Betti numbers for $G = \mathbb{Z}/2\mathbb{Z}$ by using the equivariant extension criterion (theorem \ref{crit_action}) on invariant chains, we finally highlight sufficient conditions for the Nash-constructible filtration to compute them.
\\
 
We begin this paper by recalling the principal definitions and properties from \cite{Pri-CA} about the weight complex with action that we will need. In particular, we give the equivariant extension criterion and the Smith Nash-constructible short exact sequence for an algebraic involution.

In section \ref{sect_equiv_weight_fil}, after some reminders about group cohomology and homology, we define a functor computing group homology and equivariant homology. Showing the compatibility of this functor with filtered categories, we apply it to the weight complex with action to obtain the equivariant weight complex. We then study the induced spectral sequence. We give also some examples for the computation of the equivariant weight filtration, and justify that the odd-order group case is the simple case.

Finally, in section \ref{sect_add_inv}, we explain how we obtain additive invariants from double complexes induced by the Nash-constructible filtration and group cohomology. In some cases, we recover the equivariant virtual Betti numbers. In the last paragraph, we recover all of these for $G = \mathbb{Z}/2\mathbb{Z}$, considering the invariant chains.
\\

{\bf Acknowledgements.} The author wishes to thank M. Coste, G. Fichou, F. Guill\'en, T. Limoges, C. McCrory and A. Parusi\'nski for useful discussions and comments.

\section{Weight complex with action} \label{compl_act}

In this section, we recall the important results from \cite{Pri-CA}, in particular the existence and uniqueness of the weight complex with action (theorem \ref{th_weight-complex-act}) and the exactness of what we called the Smith Nash-constructible short sequence for an involution (theorem \ref{nash-smith_ex}).

For the reader's convenience, we first recall the definition of the complex of semialgebraic chains with closed supports, before equipping it with a group action.

\subsection{Semialgebraic chains with action}

We recall the definition from \cite{MCP}, Appendix, of the semialgebraic chains with closed supports of a semialgebraic subset $X$ of the set of real points of a real algebraic variety. Here, a real algebraic variety is a reduced separated scheme of finite type over $\mathbb{R}$.

\begin{de} For all $k \geq 0$, we denote by $C_k(X)$ the quotient of the $\mathbb{Z}_2$-vector space generated by the closed semialgebraic subsets of $X$ of dimension $\leq k$, by the relations
\begin{itemize}
	\item the sum $A + B$ is equivalent to the closure $cl_X(A \div B)$ in $X$ (with respect to the strong topology) of the symmetric difference of $A$ and $B$,
	\item the class of $A$ is zero if the dimension of $A$ is strictly smaller than $k$.
\end{itemize}
A semialgebraic chain with closed supports of dimension $k$ is by definition an equivalence class of $C_k(X)$. Any chain $c$ of $C_k(X)$ can be written as the class of a closed semialgebraic subset $A$ of $X$ of dimension $\leq k$, denoted by $[A]$.

The boundary operator $\partial_k : C_k(X) \rightarrow C_{k-1}(X)$ is defined by $\partial_k c = [\partial A]$, if $c = [A] \in C_k(X)$, where $\partial A$ denotes the semialgebraic boundary $\{x \in A~|~\Lambda \mathbf{1}_A(X) \equiv 1~{\rm mod}~2\}$ of $A$ ($\Lambda$ is the link on the constructible functions, cf \cite{MCP-ACF}).
\end{de}

Consider now a group $G$ acting on $X$ by semialgebraic homeomorphisms. By functoriality of the semialgebraic chains with closed supports, an action by linear isomorphisms is induced on the complex $C_*(X)$.

The complex $C_*(X)$ equipped with this action becomes a $G$-complex, that is the induced action of $G$ on $C_*(X)$ commutes with its differential. Moreover, the semialgebraic chains with action are functorial with respect to equivariant proper continuous semialgebraic maps, and the operations of equivariant restriction, closure and pullback are equivariant.

For more details about semialgebraic chains, we refer to the appendix of \cite{MCP}, and for their equivariance, to the second section of \cite{Pri-CA}.

\subsection{Weight complex with action} \label{sect_comp_act}

Let $G$ be a finite group.

We give the steps that led us to the construction and the uniqueness of the weight complex with action in \cite{Pri-CA}, notably an equivariant version of the extension criterion \cite{GNA} Th\'eor\`eme 2.2.2 of F. Guill\'en and V. Navarro Aznar (theorem \ref{crit_action}). 

The next definition is about notations for the equivariant categories we are working in.

\begin{de} We denote by
\begin{itemize}
	\item $\mathbf{Sch}_c^G(\mathbb{R})$ the category of real algebraic varieties equipped with an action of $G$ by algebraic isomorphisms -we call such objects real algebraic $G$-varieties- and equivariant regular proper morphisms,
	\item $\mathbf{Reg}_{comp}^G(\mathbb{R})$ the subcategory of compact nonsingular $G$-varieties,
	\item $\mathbf{V}^G(\mathbb{R})$ the subcategory of projective nonsingular $G$-varieties.
\end{itemize}
We denote also by
\begin{itemize}
	\item $\mathcal{C}^G$ the category of bounded $G$-complexes of $\mathbb{Z}_2$-vector spaces equipped with an increasing bounded filtration by $G$-complexes with equivariant inclusions -we call such objects filtered $G$-complexes- and equivariant morphisms of filtered complexes,
	\item $\mathcal{D}^G$ the category of bounded $G$-complexes and equivariant morphisms of complexes.
\end{itemize}
\end{de}

By an action of $G$ by algebraic isomorphims on a real algebraic variety $X$, we mean an action by isomorphisms of schemes such that the orbit of any point in $X$ is contained in an affine open subscheme.

For any real algebraic $G$-variety $X$, we will denote by $C_*(X)$ the semialgebraic chain $G$-complex of the set of real points of $X$.
\\

In \cite{Pri-CA} Th\'eor\`eme 3.5, we defined the weight complex with action of $G$ on the category of real algebraic $G$-varieties, by equipping McCrory-Parusi\'nski's weight complex (\cite{MCP} Theorem 1.1) with the action induced by functoriality :

\begin{theo} \label{th_weight-complex-act} {\rm(\cite{Pri-CA} Th\'eor\`eme 3.5)} The functor 
$$F^{can} C_{*} : \mathbf{V}^G(\mathbb{R}) \longrightarrow H o \, \mathcal{C}^G~;~X \mapsto F^{can} C_*(X)$$
admits an extension to a functor 
$${}^G\!\mathcal{W} C_{*} : \mathbf{Sch}_c^G(\mathbb{R}) \longrightarrow H o \, \mathcal{C}^G$$
defined for all real algebraic $G$-varieties and all equivariant proper regular morphisms, which satisfies the following properties :
\begin{enumerate}
	\item Acyclicity : For any acyclic square 
\begin{equation*} \label{acyclic_square}
\begin{array}{ccc}
\ \widetilde{Y} & \rightarrow & \widetilde{X} \\
  \downarrow &      &  \downarrow_{\pi}  \\
  Y & \xrightarrow{i} & X 
  \end{array} \tag{2.1}
 \end{equation*}
in $\mathbf{Sch}_c^G(\mathbb{R})$ (that is a commutative diagram (\ref{acyclic_square}) of objects and morphisms of $\mathbf{Sch}_c^G(\mathbb{R})$ such that $i$ is an equivariant inclusion of a closed subvariety, $\widetilde{Y} = \pi^{-1}(Y)$ and the restriction $\pi : \widetilde{X} \setminus \widetilde{Y} \rightarrow X \setminus Y$ is an equivariant isomorphism), the simple filtered complex of the $\square_1^{+}$-diagram in $\mathcal{C}^G$
	$$\begin{array}{ccc}
\ {}^G\!\mathcal{W} C_{*}(\widetilde{Y}) & \rightarrow & {}^G\!\mathcal{W} C_{*}(\widetilde{X}) \\
  \downarrow &      &  \downarrow  \\
  {}^G\!\mathcal{W} C_{*}(Y) & \rightarrow & {}^G\!\mathcal{W} C_{*}(X) 
  \end{array}$$
  is acyclic (i.e. isomorphic to the zero complex in $H o \, \mathcal{C}^G$).
  	\item Additivity : For an equivariant closed inclusion $Y \subset X$, the simple filtered complex of the $\square_0^{+}$-diagram in $\mathcal{C}^G$
	$${}^G\!\mathcal{W} C_{*}(Y) \rightarrow {}^G\!\mathcal{W} C_{*}(X)$$
	is isomorphic to ${}^G\!\mathcal{W} C_{*}(X \setminus Y)$.
\end{enumerate}
Such a functor ${}^G\!\mathcal{W} C_{*}$ is unique up to a unique isomorphism of $H o \, \mathcal{C}^G$.
\end{theo}

The filtration $F^{can}$ is the canonical filtration of bounded chain complexes (see \cite{Del2} and also \cite{MCP}) and $H o \,  \mathcal{C}^G$ denotes the category $\mathcal{C}^G$ localised with respect to equivariant filtered quasi-isomorphisms, that we will sometimes call quasi-isomorphisms of $\mathcal{C}^G$, i.e. the equivariant filtered morphisms between filtered $G$-complexes that induce an equivariant isomorphism at the level $E^1$ of the induced spectral sequences.
\\

In order to show the uniqueness of the weight complex with action ${}^G\!\mathcal{W} C_{*}$, we used an equivariant version of the extension criterion \cite{GNA} Th\'eor\`eme 2.2.2 of Guill\'en and Navarro Aznar (that McCrory and Parusi\'nski used to show the uniqueness of their weight complex, see \cite{MCP} Theorem 1.1), which justifies the restriction to the case of a finite group, for which there exist an equivariant compactification, an equivariant Chow-Hironaka lemma and an equivariant resolution of singularities in the category $\mathbf{Sch}_c^G(\mathbb{R})$ (by \cite{DL} Appendix) :

\begin{theo} \label{crit_action} {\rm(\cite{Pri-CA} Th\'eor\`eme 3.5)} Let $\mathcal{A}$ be a category of cohomological descent and  
$$F : \mathbf{V}^G(\mathbb{R}) \longrightarrow H o \, \mathcal{A}$$
be a contravariant $\Phi$-rectified functor verifying 
\begin{itemize}
\item[(F1)] $F(\emptyset) = 0$, and the canonical morphism $F(X \sqcup Y) \rightarrow F(X) \times F(Y)$ is an isomorphism (in $H o \, \mathcal{A}$),
\item[(F2)] if $X_{\bullet}$ is an elementary acyclic square of $\mathbf{V}^G(\mathbb{R})$, then $\mathbf{s}F(X_{\bullet})$ is acyclic.
\end{itemize}
Then, there exists an extension of $F$ to a contravariant $\Phi$-rectified functor
$$F_c : \mathbf{Sch}_c^G(\mathbb{R}) \rightarrow H o \, \mathcal{A}$$
such that :
\begin{enumerate}
	\item if $X_{\bullet}$ is an acyclic square of $\mathbf{Sch}_c^G(\mathbb{R})$, then $\mathbf{s}F_c(X_{\bullet})$ is acyclic,
	\item if $Y$ is a closed subvariety of $X$ stable under the action of $G$ on $X$, we have a natural isomorphism (in $H o \, \mathcal{A}$)
$$F_c(X \setminus Y) \cong \mathbf{s} (F_c(X) \rightarrow F_c(Y)).$$
\end{enumerate}
Furthermore, this extension is unique up to a unique isomorphism.
\end{theo}

\begin{rem} \label{rem_fil_geom_act}
\begin{itemize} 
	\item For every real algebraic $G$-variety $X$, we have an equivariant isomorphism $H_n({}^G\!\mathcal{W} C_*(X)) \cong H_n(X)$ for all $n \in \mathbb{Z}$.
	\item The weight filtration and the weight spectral sequence are equipped with the action of $G$ induced by the $G$-action on the weight complex. 
	\item McCrory and Parusi\'nski's geometric/Nash-constructible filtration on real algebraic varieties (\cite{MCP} sections 2 and 3), equipped with the action of $G$ induced by functoriality, realizes the weight complex with action.
	 \item If $X$ is a compact nonsingular $G$-variety, ${}^G\!\mathcal{W} C_*(X)$ is isomorphic to $F^{can} C_*(X)$ in $H o \, \mathcal{C}^G$, and the same goes for all the realizations of the weight complex with action, in particular for the geometric/Nash-constructible filtration with action.	
\end{itemize}

\end{rem}

For more details about the weight complex with action, we refer to \cite{Pri-CA} section 3.
\\

In the following, the weight complex with action will be simply denoted by $\mathcal{W} C_*$ if the context is explicit.

\subsection{The Smith Nash-constructible exact sequence in the case $G~=~\mathbb{Z}/2 \mathbb{Z}$} \label{ex_seq_smith_nash_cons}

In the last paragraph of this section, we recall a result from \cite{Pri-CA} section 5 : one can use the Nash-constructible filtration to implement a notion of regularity into the Smith short exact sequence of a real algebraic variety equipped with an involution.

More precisely, let $X$ be a real algebraic variety equipped with an algebraic involution $\sigma$, then any invariant chain can be split with some regularity into two parts exchanged by the action of $G := \{1, \sigma\} = \mathbb{Z}/2\mathbb{Z}$ (modulo the restriction to the fixed points set of $X$) :

\begin{theo} \label{nash-smith_ex} {\rm(\cite{Pri-CA} Th\'eor\`eme 5.5)} Let $c \in (\mathcal{N}_{\alpha} C_k(X))^G$ be a chain of $X$ of dimension $k$ and of index $\alpha$ with respect to the Nash-constructible filtration. Then there exists $c' \in \mathcal{N}_{\alpha + 1} C_k(X)$ such that
$$c = c_{|X^G} + (1 + \sigma) c'$$
(the restriction $c_{|X^G}$ is in $\mathcal{N}_{\alpha} C_k(X^G)$).

Consequently, for all $\alpha$, the short sequence of complexes
$$0 \rightarrow \mathcal{N}_{\alpha} C_*(X^G) \oplus (1+ \sigma) T^{\alpha + 1}_*(X) \rightarrow \mathcal{N}_{\alpha} C_*(X) \rightarrow (1 + \sigma) \mathcal{N}_{\alpha} C_*(X) \rightarrow 0,$$
where $T^{\alpha + 1}_k(X) := \{c \in \mathcal{N}_{\alpha+1} C_k(X)~|~(1 + \sigma) c \in \mathcal{N}_{\alpha} C_k(X) \}$, is exact. We call this sequence the Smith Nash-constructible exact sequence of $X$ of index $\alpha$.
\end{theo}

The last property we recall is an interpretation of the Smith Nash-constructible exact sequence when the variety $X$ is compact and the action of $G$ is free. In this case, the quotient of the set of real points of $X$, denoted by $X_{\mathbb{R}}$, by the action of $G$ is an arc-symmetric set and the invariant chains of $X$ correspond to the chains of the quotient~:

\begin{prop} \label{fil_Nash_quotient} {\rm(\cite{Pri-CA} Proposition 5.6)} Let $X$ be a compact real algebraic variety equipped with a fixed-point free action of $G = \mathbb{Z}/2\mathbb{Z}$. Then the morphism of filtered complexes 
$$\left(\mathcal{N} C_*(X)\right)^G \rightarrow \mathcal{N} C_*\left(X_{\mathbb{R}}/G\right),$$
induced by the quotient map $\pi : X_{\mathbb{R}} \rightarrow X_{\mathbb{R}}/G$, is a filtered isomorphism (recall that by definition $\mathcal{N} C_*(X) = \mathcal{N} C_*(X_{\mathbb{R}})$).
\end{prop}

\section{Equivariant weight filtration for real algebraic $G$-varieties} \label{sect_equiv_weight_fil}

Let $G$ be a finite group.

In this section, we construct a weight filtration on the equivariant homology of real algebraic $G$-varieties defined by J. van Hamel in \cite{VH}, applying on the weight complex with action some functor which computes this equivariant homology when applied to the complex of semialgebraic chains with closed supports (definition and proposition \ref{fonct_L}, definition \ref{hom_equiv}).

Significative differences will appear between the equivariant weight complex and the weight complex. In particular, unlike McCrory and Parusi\'nski's weight spectral sequence (\cite{MCP} section 1.3), the equivariant weight spectral sequence may not be bounded (example \ref{ex_suite_spec_poids_equiv}) and may not degenerate at level two in the compact nonsingular case (proposition \ref{var_comp-non-sing_equiv}, example \ref{ex_sphere_act}). 

Nevertheless, equivariant homology will provide us other spectral sequences (trivial in the non-equivariant case) useful to understand the equivariant geometry involved in the equivariant weight spectral sequence (proposition \ref{eq_noneq_weight_spec}).

\subsection{Group (co)homology and the functor $L$} \label{group_cohom_l_func}

In the first paragraph, we recall the basic background about group (co)homology with coefficients in a module and a chain complex, which we will use to define the equivariant homology of a real algebraic variety with action.
\\

If $M$ is a $\mathbb{Z}[G]$-module, the $n^{\text{th}}$ group of cohomology of the group $G$ with coefficients in $M$ is given by
$$H^n(G,M) := \Ext_{\mathbb{Z}[G]}^n(\mathbb{Z},M) = H^n(\Hom_{\mathbb{Z}[G]}(F_*,M)),$$
where $F_*$ is a projective resolution of $\mathbb{Z}$ by $\mathbb{Z}[G]$-modules.

\begin{rem}
\begin{itemize}
	\item For $n <0$, $H^n(G, M) = 0$.
	\item If $k$ is a ring and if $M$ is a $k[G]$-module, then for all $n \in \mathbb{Z}$, 
$$H^n(G,M) = \Ext_{\mathbb{Z}[G]}^n(\mathbb{Z},M) \cong \Ext_{k[G]}^n(k,M),$$ 
by an isomorphism of $k$-modules. In the rest of this paper, we will be considering $\mathbb{Z}_2$-vector spaces equipped with a linear action of $G$, so there will be no difference in considering a projective resolution of $\mathbb{Z}$ by $\mathbb{Z}[G]$-modules or a projective resolution of $\mathbb{Z}_2$ by $\mathbb{Z}_2[G]$-modules.
\end{itemize}

\end{rem}

\begin{ex} (\cite{Bro}) Let $G$ a finite cyclic group of order $d$ generated by $\sigma$. Then, if we denote $N := \sum_{1 \leq i \leq d} \sigma^i$, a projective resolution of $\mathbb{Z}$ by $\mathbb{Z}[G]$-modules is given by
$$\cdots \longrightarrow \mathbb{Z}[G] \xrightarrow{\sigma - 1} \mathbb{Z}[G] \xrightarrow{N} \mathbb{Z}[G] \xrightarrow{\sigma - 1} \mathbb{Z} \rightarrow 0,$$
and the cohomology of $G$ with coefficients in a $\mathbb{Z}[G]$-module $M$ by
$$H^n(G,M) =
\begin{cases}
\frac{M^G}{N M} \text{ if $n$ is an even positive integer,} \\
\frac{ker~(M \xrightarrow{N} M)}{(\sigma - 1) M} \text{ if $n$ is an odd positive integer,}\\
M^G \text{ if $n = 0$.} 
\end{cases}$$
\end{ex}

\begin{ex} \label{group_cohom_z2z} If $G = \mathbb{Z}/2\mathbb{Z}$ and if $M$ is a $\mathbb{Z}_2[G]$-module, the cohomology of $G$ with coefficients in $M$ is
$$H^n(G,M) =
\begin{cases}
M^G \text{ si $n = 0$,} \\
\frac{M^G}{(1+\sigma) M} \text{ si $n > 0$.}
\end{cases}$$
\end{ex}

For more details and background about group cohomology, see for instance \cite{Bro} or \cite{CTVZ}.
\\

Now we define the homology of $G$ with coefficients in a chain $G$-complex of $\mathbb{Z}_2$-vector spaces, using a functorial operation denoted by $L$. We denote by $\mathcal{D}_-$ the category of bounded above chain complexes of $\mathbb{Z}_2$-vector spaces, and $H o \, \mathcal{D}_-$ the category $\mathcal{D}_-$ localised with respect to quasi-isomorphisms.

\begin{defprop} \label{fonct_L} Let $K_*$ be in $\mathcal{D}^G$. Let $... \rightarrow F_2 \xrightarrow{\Delta_2} F_1 \xrightarrow{\Delta_1} F_0 \rightarrow \mathbb{Z} \rightarrow~0$ be a resolution of $\mathbb{Z}$ by projective $\mathbb{Z}[G]$-modules. 

Then the complex $L_*(K_*)$ is defined as the total complex associated to the double complex
$$(\Hom_G(F_{-p},K_q))_{p,q \in \mathbb{Z}}.$$ 

The operation $L : \mathcal{D}^G \rightarrow \mathcal{D}_-~;~K_* \mapsto L_*(K_*)$ is functorial.
\end{defprop}

\begin{rem} For $G = \{e\}$, considering $... \rightarrow 0 \xrightarrow{} 0 \xrightarrow{} \mathbb{Z} \xrightarrow{id} \mathbb{Z} \rightarrow 0$ as a projective resolution, we obtain $L_*(K_*) = K_*$.
\end{rem}

The homology of the group $G$ with coefficients in a $G$-complex $K_*$ is then defined as the homology of the complex $L_*(K_*)$ and denoted by $H_*(G, K_*)$.

The two spectral sequences associated to the double complex $(Hom(F_{-p},C_q))_{p,q \in \mathbb{Z}}$ converge to this homology :
$$\left. \begin{matrix}
{}_{I}\!E^2_{p,q} & = & H^{-p}(G,H_q(K_*)) \\
{}_{II}\!E^1_{p,q} & = & H^{-p}(G,K_q)
\end{matrix} \right\} \Longrightarrow H_{p+q}(G,K_*).$$ 
The first spectral sequence is called the Hochschild-Serre spectral sequence associated to the group $G$ and the $G$-complex $K_*$. Notice that, in particular, since the group cohomology with coefficients in a module does not depend on the considered projective resolution (of $\mathbb{Z}$ by $\mathbb{Z}[G]$-modules or of $\mathbb{Z}_2$ by $\mathbb{Z}_2[G]$-modules), neither do the group homology with coefficients in a $G$-complex of $\mathcal{D}^G$ and the functor $\mathcal{D}^G \rightarrow H o \, \mathcal{D}_-~;~K_* \mapsto L_*(K_*)$ (also denoted by $L$ for convenience).

The Hochschild-Serre spectral sequence also allows to prove that $L$ preserves quasi-isomorphisms~:

\begin{prop} An equivariant quasi-isomorphism $f : K_* \rightarrow M_*$ induces an isomorphism $H_*(G,K_*) \rightarrow H_*(G,M_*)$.
\end{prop}

\begin{proof} The equivariant quasi-isomorphism $f$ induces an isomorphism from the level ${}_{I}\!E^2$ of the induced Hochschild-Serre spectral sequences, which converge to the homologies of $G$ with coefficients in $K_*$ and $M_*$ respectively. 
\end{proof}

We also denote by $L$ the induced functor $H o \, \mathcal{D}^G \rightarrow H o \, \mathcal{D}_-$.
\\

Furthermore, we show that if we apply the functor $L$ to a complex of $\mathcal{C}^G$, we can obtain a filtered complex, and that this operation preserves filtered quasi-isomorphisms.
\\

Let $\mathcal{C}_-$ denote the category of bounded above complexes of $\mathbb{Z}_2$-vector spaces equipped with an increasing bounded filtration, and morphisms of filtered complexes.

\begin{prop}  Let $(K_*,J)$ be in $\mathcal{C}^G$. The equivariant increasing bounded filtration $J$ of the $G$-complex $K_*$ induces an increasing bounded filtration $\mathcal{J}$ on the complex $L_*(K_*)$, defined by 
$$\mathcal{J}_{\alpha}L_k(K_*) := L_k(J_{\alpha} K_*).$$

In $H o \, \mathcal{C}_-$, the couple $(L_*(K_*), \mathcal{J})$ is independent from the chosen projective resolution.
\end{prop}

\begin{proof} We show that two projective resolutions induce quasi-isomorphic filtered complexes in~$\mathcal{C}_-$.

Let $(F_i)_i$ and $(F'_j)_j$ be two resolutions of $\mathbb{Z}$ or $\mathbb{Z}_2$ by projective $\mathbb{Z}[G]$-modules, respectively $\mathbb{Z}_2[G]$-modules. We denote by $\mathcal{J} L_*(K_*)$ and $\mathcal{J}' L'_*(K_*)$ the respectively associated filtered complexes, and by $E^r$ and $E'^r$ the respectively induced spectral sequences. 

We have
$$E^0_{p,q} = \frac{\mathcal{J}_p L_{p+q}}{\mathcal{J}_{p-1} L_{p+q}} = \frac{\bigoplus_{a+b = p+q} \Hom_G(F_{-a},J_p K_b)}{\bigoplus_{a+b = p+q} \Hom_G(F_{-a},J_{p-1} K_b)} = \bigoplus_{a+b = p+q} \Hom_G \left(F_{-a}, \frac{J_p K_b}{J_{p-1} K_b}\right)$$
because the $G$-modules $F_i$ are projective. Then we have $E^0_{p,*} = L_{p+*}\left(\frac{J_p K_*}{J_{p-1} K_*}\right)$ and $E^{'0}_{p,*} = L'_{p+*}\left(\frac{J_p K_*}{J_{p-1} K_*}\right)$ for all $p \in \mathbb{Z}$. The homologies of these two complexes give the homology of the group~$G$ with coefficients in the complex $\frac{J_p K_*}{J_{p-1} K_*}$, inducing an isomorphism between $E^1$ et $E'^{1}$.

\end{proof}

We use this computation to show that the functor $L$ preserves filtered quasi-isomorphisms~:

\begin{prop} The operation $L$ induces a functor
$$H o \, \mathcal{C}^G \rightarrow H o \, \mathcal{C}_{-}~;~(K_*,J) \mapsto (L_*(K_*), \mathcal{J}),$$
that we also denote by $L$.
\end{prop}

\begin{proof} Let $\varphi : K_* \rightarrow M_*$ be a quasi-isomorphism in $\mathcal{C}^G$ and $J$ and $I$ be the respective filtrations of $K_*$ and $M_*$. We have 
$$E^{0~(L_*(K_*))}_{p,q} = L_{p+q}\left(E^{0~(K_*)}_{p,*-p}\right) \mbox{ and } E^{0~(L_*(M_*))}_{p,q} = L_{p+q}\left(E^{0~(M_*)}_{p,*-p}\right).$$ 
But, for all $p \in \mathbb{Z}$, the morphism $\varphi : K_* \rightarrow M_*$ induces an equivariant quasi-isomorphism $E^{0~(K_*)}_{p,*-p} \rightarrow E^{0~(M_*)}_{p,*-p}$, and then we use the fact that $L$ preserves quasi-isomorphisms to conclude. 
\end{proof}

We then apply $L$ to the canonical filtration. The induced spectral sequence coincides with the Hochschild-Serre spectral sequence : 

\begin{lem} \label{lem_spec_seq_fil_can_equiv} Let $(K_*,\partial_*)$ be a $G$-complex equipped with the canonical filtration $F^{can}$. We consider the induced filtered complex $\mathcal{F}^{can} L_*(K_*)$, and denote by $E$ the induced spectral sequence. 

Then we have, for all $p,q \in \mathbb{Z}$, 
$$E^1_{p,q} = H^{-2p-q}(G,H_{-p}(K_*)),$$
and, for all $r \geq 1$,
$$E^r_{p,q} = {}_{I}\!E^{r+1}_{2p+q ,-p}.$$
\end{lem}

\begin{proof} Let $(F_*,\Delta_*)$ be a projective resolution over $\mathbb{Z}[G]$ of $\mathbb{Z}$. Let $p \in \mathbb{Z}$.

By a direct computation and using the fact that the $F_i$'s are projective $\mathbb{Z}[G]$-modules, one can see that the complex $E^0_{p,*-p}$ is the mapping cone of the morphism
$$\phi : \Hom_G\left(F_{-(p+*)}, K_{-p+1} / ker \, \partial_{-p+1}\right) \rightarrow \Hom_G\left(F_{-(p+*)}, ker \, \partial_{-p}\right).$$
On the other hand, because the $F_i$'s are projective $\mathbb{Z}[G]$-modules, the short sequence of complexes
$$0 \rightarrow \Hom_G\left(F_{-(p+*)}, K_{-p+1} / ker \, \partial_{-p+1}\right) \rightarrow \Hom_G\left(F_{-(p+*)}, ker \, \partial_{-p}\right) \rightarrow \Hom_G \left(F_{-(p+*)}, H_{-p}(K_*) \right) \rightarrow 0$$
is exact.

Considering now the induced long exact sequences in homology and using the five femma, we obtain that, for all $p,q \in \mathbb{Z}$, 
$$E^1_{p,q} = H_{p+q} \left(E^0_{p,*-p}\right) \cong  H_{p+q} \left(\Hom_G\left(F_{-(p+*)}, H_{-p}(K_*)\right) \right) = H^{-2p -q}(G,H_{-p}(K_*)).$$

The spectral sequence $E$ and the Hochschild-Serre spectral sequence are then naturally isomorphic (modulo reindexing) because of the naturality of the above isomorphism and because the differentials of both spectral sequences are induced by the same morphisms.
\end{proof}

We end this paragraph with an essential point : the functor $L$ commutes with the operation which associates to any cubical diagram in $\mathcal{C}^G$ (resp. $\mathcal{C}_-$) its simple filtered complex. For the definition of a cubical diagram and the associated simple filtered complex, see \cite{MCP} section 1. 

\begin{prop} \label{l_com_s} If $\mathcal{K}$ is a cubical diagram of type $\square_n^+$ in $\mathcal{C}^G$, then $\mathbf{s}(L_*(\mathcal{K})) = L_*(\mathbf{s}\mathcal{K})$ (in~$\mathcal{C}_-$ if we consider the same projective resolution on both sides).
\end{prop}

\begin{proof}  Direct computation.
\end{proof}

\subsection{Equivariant weight complex}

Applying the functor $L$ to the weight complex with action (theorem \ref{th_weight-complex-act}), we obtain the equivariant weight complex which will induce a weight filtration on the equivariant homology we define in definition \ref{hom_equiv} below.

\begin{de} \label{hom_equiv} Let $X$ be a real algebraic $G$-variety. We denote $C_*^G(X) := L_*(C_*(X))$ and, for $n \in \mathbb{Z}$, we associate to $X$ 
$$H_n(X ; G) := H_n(G,C_*(X)) = H_n(C^G_*(X)), $$
its $n^{\text{th}}$ equivariant homology group, where $C_*(X)$ is the $G$-complex of semialgebraic chains with closed supports of the set of real points of $X$.
\end{de}

\begin{rem} 
\begin{itemize} 
	\item The Hochschild-Serre spectral sequence 
	$${}_{I}\!E^2_{p,q} = H^{-p}(G,H_q(X)) \Rightarrow H_{p+q}(X ; G)$$	allows one to understand this equivariant homology from a geometrical point of view : it involves the geometry of the considered real algebraic $G$-variety, the geometry of the action and the geometry of the group $G$ itself.
\item This equivariant homology is the same as J. van Hamel's equivariant homology defined in \cite{VH} Chapter III Definition 1.2 (at least for compact real algebraic $G$-varieties).  
	\item For $G = \{e\}$, $H_n(X;G) = H_n(X)$.
\end{itemize}
\end{rem}

\begin{ex} \label{ex_sphere_act} We use the Hochschild-Serre spectral sequence to compute the equivariant homology of the $2$-dimensional sphere $X$ given by the equation $x^2 + y^2 + z^2 = 1$ in $\mathbb{R}^3$, equipped with the action of $G =~\mathbb{Z}/2\mathbb{Z}$ given by $\sigma : (x,y) \mapsto (-x,y)$.

The page ${}_{I}\!E^2(X)$ is 
$$\begin{matrix}
\ldots & \mathbb{Z}_2[X] & \ldots & \mathbb{Z}_2[X] & \mathbb{Z}_2[X] \\
\ldots & 0 & \dots & 0 & 0 \\ 
\ldots & \mathbb{Z}_2[\{p_0\}] & \ldots & \mathbb{Z}_2[\{p_0\}] & \mathbb{Z}_2[\{p_0\}]
\end{matrix}$$
where $p_0$ is a point of $X$, that we choose in the fixed points set $X^G$. Then, since the differentials of the rows of the double complex inducing the spectral sequence are $1 + \sigma$, we have ${}_{I}\!E^2(X) = {}_{I}\!E^{\infty}(X)$ and 
$$H_k(X; G) = 
\begin{cases}
 \mathbb{Z}_2[X] \text{ if $k = 1$ or $2$,} \\
  \mathbb{Z}_2[X] \oplus  \mathbb{Z}_2[\{p_0\}] \text{ if $k \leq 0$.}
\end{cases}$$

If we consider now the fixed point free $G$-action on $X$ given by $\sigma : (x,y) \mapsto (-x,-y)$, the page ${}_{I}\!E^2(X)$ is the same (with $p_0$ being any point of $X$), but the differentials $d^3$ are not trivial (we have $d^3([\{p_0\}]) = [X]$). Consequently, with respect to this action, the Hochschild-Serre spectral sequence degenerates at page ${}_{I}\!E^4(X)$ :
$$\begin{matrix}
\ldots & 0 & \ldots & 0 & \mathbb{Z}_2[X] & \mathbb{Z}_2[X] & \mathbb{Z}_2[X] \\
\ldots & 0 & \dots & 0 & 0 & 0 & 0 \\ 
\ldots & 0 & \ldots & 0 & 0 & 0 & 0
\end{matrix}$$
and
$$H_k(X; G) = 
\begin{cases}
 \mathbb{Z}_2[X] \text{ if $k = 0,1$ or $2$,} \\
  0 \text{ if $k < 0$.}
\end{cases}$$
\end{ex}

The second spectral sequence 
$${}_{II}\!E^1_{p,q} = H^{-p}(G,C_q(X)) \Rightarrow H_{p+q}(X ; G)$$
associated to the equivariant homology can also be useful, as in the following case :

\begin{lem} \label{lem_formula_hom_equiv} Let $G = \mathbb{Z}/2\mathbb{Z}$. Then, for all real algebraic $G$-varieties $X$ and all $k \in \mathbb{Z}$,
$$H_k(X ; G) =  \left(ker \, \partial_{k}\right)^G/ \partial_{k+1}\left(\left(C_{k+1}(X)\right)^G\right) \oplus \bigoplus_{i \geq k+1} H_i\left(X^G\right).$$
\end{lem}

\begin{proof} The equivariant homology is the homology of the total complex associated to the double complex
$$\begin{array}{cccccc}
\ \longleftarrow & C_d(X) & \xleftarrow{1 + \sigma} & C_d(X) & \xleftarrow{1 + \sigma} & C_d(X) \\
& \downarrow \partial_d & &  \downarrow \partial_d&  & \downarrow \partial_d \\
\longleftarrow & C_{d-1}(X) & \xleftarrow{1 + \sigma} & C_{d-1}(X) & \xleftarrow{1 + \sigma} & C_{d-1}(X) \\
& \downarrow \partial_{d-1} & &  \downarrow \partial_{d-1} & &\downarrow \partial_{d-1} \\
& \vdots & & \vdots & & \vdots \\
& \downarrow \partial_2 & &  \downarrow \partial_2 & & \downarrow \partial_2 \\
\ \longleftarrow & C_1(X) & \xleftarrow{1 + \sigma} & C_1(X) & \xleftarrow{1 + \sigma} & C_1(X)\\
& \downarrow \partial_1 & &  \downarrow \partial_1 & & \downarrow \partial_1 \\
\longleftarrow & C_{0}(X) & \xleftarrow{1 + \sigma} & C_{0}(X) & \xleftarrow{1 + \sigma} & C_{0}(X) \\
\end{array}$$
where $d$ is the dimension of $X$. Using the Smith short exact sequence 
$$0 \rightarrow C_*\left(X^G\right) \oplus (1+\sigma) C_*(X) \rightarrow C_*(X) \rightarrow (1 + \sigma) C_*(X) \rightarrow 0$$
we compute the level one of the spectral sequence ${}_{II}\!E$ :
$$\begin{array}{ccccccc}
\ \cdots &  & C_d\left(X^G\right) &  & C_d\left(X^G\right)&  & \left(C_d(X)\right)^G \\
 & & \downarrow \partial_d & &  \downarrow \partial_d&  & \downarrow \partial_d \\
\cdots &  & C_{d-1}\left(X^G\right) &  & C_{d-1}\left(X^G\right) & &\left(C_{d-1}(X)\right)^G  \\
& &\downarrow \partial_{d-1} & &  \downarrow \partial_{d-1} & &\downarrow \partial_{d-1} \\
&& \vdots & & \vdots & & \vdots \\
& &\downarrow \partial_2 & &  \downarrow \partial_2 & & \downarrow \partial_2 \\
\ \cdots &  & C_1\left(X^G\right) &  & C_1\left(X^G\right) & &\left(C_1(X)\right)^G \\
& &\downarrow \partial_1 & &  \downarrow \partial_1 & & \downarrow \partial_1 \\
 \cdots & & C_{0}\left(X^G\right) & & C_{0}\left(X^G\right) & &\left(C_{0}(X)\right)^G \\
\end{array}$$

The page ${}_{II}\!E^2$ is 
$$\begin{array}{cccccc}
\ \cdots & H_d\left(X^G\right) &  & H_d\left(X^G\right) &  &  \left(ker \, \partial_d\right)^G \\
& & &  &  & \\
\cdots & H_{d-1}\left(X^G\right) &  & H_{d-1}\left(X^G\right) &  & \left(ker \, \partial_{d-1}\right)^G/ \partial_d\left(\left(C_d(X)\right)^G\right) \\
& &  & & & \\
& \vdots & & \vdots & & \vdots \\
& &  & &  & \\
\ \cdots & H_1\left(X^G\right) &  & H_1\left(X^G\right) &  &  \left(ker \, \partial_{1}\right)^G/ \partial_2\left(\left(C_2(X)\right)^G\right) \\
& &  & & & \\
 \cdots & H_{0}\left(X^G\right) & & H_{0}\left(X^G\right) & &  \left(C_{0}(X)\right)^G /  \partial_1\left(\left(C_1(X)\right)^G\right) \\
\end{array}$$
and the differentials ${}_{II}\!d^r$ for $r \geq 2$ all vanish because any element of ${}_{II}\!E^r_{p,q}$ with $p < 0$ and $0 \leq q \leq d$ can be represented by a cycle. Consequently, the spectral sequence ${}_{II}\!E$ degenerates at level two and, since it converges to the equivariant homology of $X$, we have
$$H_k(X ; G) =  (ker \, \partial_{k})^G/ \partial_d((C_{k+1}(X))^G) \oplus \bigoplus_{i \geq k+1} H_i(X^G)$$
for all $k \in \mathbb{Z}$.
\end{proof}
 
In definition \ref{def_eq_we_co} below, we define the equivariant weight complex. Recall that the weight complex with action has values in the localised category $H o \, \mathcal{C}^G$. As a consequence, we can apply the functor $L : H o \, \mathcal{C}^G \rightarrow H o \, \mathcal{C}_-$.

\begin{de} \label{def_eq_we_co} Let $X$ be a real algebraic $G$-variety. We denote 
$$\Omega C^G_*(X) := L_*(\mathcal{W} C_*(X)) \in H o \, \mathcal{C}_-,$$
and we call this filtered complex the equivariant weight complex of $X$.
\end{de}

For any real algebraic $G$-variety $X$, we denote $\mathcal{F}^{can} C^G_*(X) := L_*(F^{can} C_*(X))$. The equivariant weight complex is the unique acyclic and additive extension to $\mathbf{Sch}_c^G(\mathbb{R})$ of the functor $\mathcal{F}^{can} C^G_* : \mathbf{V}^G(\mathbb{R}) \rightarrow H \circ \mathcal{C}_-$, in the following meaning :

\begin{theo}  \label{equiv_wei_comp} The operation
$$\Omega C^G_* : \mathbf{Sch}_c^G(\mathbb{R}) \rightarrow H o \, \mathcal{C}_-~;~X \mapsto \Omega C^G_*(X)$$
is a functor, which extends the functor 
$$\mathcal{F}^{can} C^G_* : \mathbf{V}^G(\mathbb{R}) \rightarrow H o \, \mathcal{C}_-~;~X \mapsto \mathcal{F}^{can} C^G_*(X),$$
to $\mathbf{Sch}_c^G(\mathbb{R})$ and verifies the following properties :
\begin{enumerate}
	\item Acyclicity : For any acyclic square \eqref{acyclic_square} in $\mathbf{Sch}_c^G(\mathbb{R})$, the simple filtered complex of the $\square_1^{+}$-diagram in $\mathcal{C_-}$
	$$\begin{array}{ccc}
\ \Omega C_{*}^G(\widetilde{Y}) & \rightarrow & \Omega C_{*}^G(\widetilde{X}) \\
  \downarrow &      &  \downarrow  \\
  \Omega C_{*}^G(Y) & \rightarrow & \Omega C_{*}^G(X) 
  \end{array}$$
  is acyclic.
  	\item Additivity : For any equivariant closed inclusion $Y \subset X$, the simple filtered complex of the $\square_0^{+}$-diagram in $\mathcal{C}_-$
	$$\Omega C_{*}^G(Y) \rightarrow \Omega C_{*}^G(X)$$
	is isomorphic to $\Omega C_{*}^G(X \setminus Y)$ in $H o \, \mathcal{C}_-$.
\end{enumerate}
Furthermore, any other functor $\mathbf{Sch}_c^G(\mathbb{R}) \rightarrow H \circ \mathcal{C}_-$ verifying these three properties is isomorphic to $\Omega C^G_*$ in $H o \, \mathcal{C}_-$, via a unique quasi-isomorphism of $\mathcal{C}_-$.
\end{theo}

\begin{proof} The equivariant weight complex $\Omega C^G_*$ is the composition of the weight complex with action ${}^G\!\mathcal{W} C_*$ with the functor $L : H o \, \mathcal{C}^G \rightarrow H o \, \mathcal{C}_-$. In particular, it is an extension to $\mathbf{Sch}_c^G(\mathbb{R})$ of the composition of the functors $F^{can} C_{*} : \mathbf{V}^G(\mathbb{R}) \longrightarrow H o \, \mathcal{C}^G$ and $L : H o \, \mathcal{C}^G \rightarrow H o \, \mathcal{C}_-$. 

Moreover, it verifies the acyclicity and additivity properties because so does the weight complex with action (theorem \ref{th_weight-complex-act}) and because the functor $L$ commutes with the operation $\mathbf{s}$ (proposition \ref{l_com_s}).
\\

Now, we show the uniqueness of the equivariant weight complex with respect to these properties, using the theorem \ref{crit_action} applied to the category $\mathcal{C}_-$ and the functor $\mathcal{F}^{can} C^G_* : \mathbf{V}^G(\mathbb{R}) \rightarrow H o \, \mathcal{C}_-$.

Indeed, 
\begin{itemize}
	\item The category $\mathcal{C}_-$ is a category of homological descent (\cite{GNA} Propri\'et\'e (1.7.5)). 	
	\item The functor $\mathcal{F}^{can} C^G_* : \mathbf{V}^G(\mathbb{R}) \rightarrow H o \, \mathcal{C}_-$ is $\Phi$-rectified because, fixing a projective resolution of $\mathbb{Z}$ by $\mathbb{Z}[G]$-modules (of $\mathbb{Z}_2$ by $\mathbb{Z}_2[G]$-modules), it is defined on $\mathcal{C}_-$. 
	\item It verifies condition (F1) : the functor $F^{can} C_* : \mathbf{V}^G(\mathbb{R}) \rightarrow H o \, \mathcal{C}^G$ verifies (F1) in $H o \, \mathcal{C}^G$ and, for any bounded $G$-complexes $K_*$ and $M_*$, we have $L_*(K_* \oplus M_*) = L_*(K_*) \oplus L_*(M_*)$.
	\item It verifies condition (F2) :  the functor $F^{can} C_* : \mathbf{V}^G(\mathbb{R}) \rightarrow H o \, \mathcal{C}^G$ verifies (F2) in $H o \, \mathcal{C}^G$ and the functor $L$ commutes with $\mathbf{s}$. 
\end{itemize}
\end{proof}

The homology of the equivariant weight complex of a real algebraic $G$-variety $X$ is the equivariant homology of $X$. Indeed, let us consider the forgetful functor of the filtration $\mathcal{C}_- \rightarrow \mathcal{D}_-$, which induces a functor $\varphi_- : H o \, \mathcal{C}_- \rightarrow H o \, \mathcal{D}_-$. We have
$$\varphi_- \circ \Omega C^G_* = L \circ ( \varphi^G \circ {}^G\!\mathcal{W} C_*)$$
(where $\varphi^G : H o \, \mathcal{C}^G \rightarrow H o \, \mathcal{D}^G$ is the functor induced by the forgetful functor $\mathcal{C}^G \rightarrow \mathcal{D}^G$). 

Since the functor $\varphi^G \circ {}^G\!\mathcal{W} C_*$ is quasi-isomorphic to $C_*$ in $\mathcal{D}^G$ (see remark \ref{rem_fil_geom_act} and \cite{Pri-CA}, Remarque 3.9), the complex $\varphi_- ( \Omega C^G_*(X)) = L_*( \varphi^G ({}^G\!\mathcal{W} C_*(X)))$ is quasi-isomorphic to $L_*(C_*(X)) = C_*^G(X)$ (the functor $L$ preserves quasi-isomorphisms) and, for all $n \in \mathbb{Z}$,
$$H_n(\Omega C^G_*(X)) = H_n(X ; G).$$

Consequently, the equivariant weight complex induces a filtration on the equivariant homology of real algebraic $G$-varieties, that we call the equivariant weight filtration. We denote it by~$\Omega$.

\subsection{Equivariant weight spectral sequence(s)} \label{subsec_equi_wei_spec}

The equivariant weight complex induces a spectral sequence, that we call the equivariant weight spectral sequence and denote by $\{ {}^G\!E^r, {}^G\!d^r \}$, which converges to the equivariant weight filtration. We reindex it, setting, as in \cite{MCP} section 1.3, 
$$p' = 2p+q,~q' = -p,~r' = r + 1,$$
and we denote by ${}^G\!\widetilde{E}^{r'}_{p',q'}$ the reindexed equivariant weight spectral sequence.

As in the non-equivariant framework, we can read the acyclicity and additivity conditions of the equivariant weight complex on the equivariant weight spectral sequence : for example, if $Y \subset X$ is an equivariant closed inclusion, we have a long exact sequence 
\begin{equation*} \label{long_ex_seq_equiv_spec_seq} ... \rightarrow {}^G\!\widetilde{E}^2_{p,q}(Y) \rightarrow {}^G\!\widetilde{E}^2_{p,q}(X) \rightarrow {}^G\!\widetilde{E}^2_{p,q}(X \setminus Y) \rightarrow {}^G\!\widetilde{E}^2_{p-1,q}(Y) \rightarrow... \tag{3.1} \end{equation*}
for all $q \in \mathbb{Z}$ (the proof runs as in the non-equivariant case : see \cite{MCP} section 1.3). Nevertheless, there are some significative differences between the weight spectral sequence and the equivariant weight spectral sequence. In particular, the equivariant weight spectral sequence is not left bounded (as a consequence, the additivity long exact sequences are not bounded either) and, on compact nonsingular varieties, it does not degenerate at level two in general (proposition \ref{var_comp-non-sing_equiv} and example \ref{ex_sphere_act}).

In proposition \ref{eq_noneq_weight_spec} below, we express the level two of the equivariant weight spectral sequence as the homology of the group $G$ with coefficients in the weight spectral sequence. From this fact, we deduce bounds for the equivariant weight spectral sequence (corollary \ref{equiv_weight_spec_bounds}) and equivariant weight filtration (corollary \ref{bounds_weight_fil}). Furthermore, one can then extract from the equivariant weight spectral sequence two other spectral sequences. We will recover finite long exact sequences of additivity from one of these in section \ref{sect_add_inv}.

\begin{prop} \label{eq_noneq_weight_spec} For all $p,q$ in $\mathbb{Z}$,
$${}^G\!\widetilde{E}^2_{p,q} = H_p\left(G, \widetilde{E}^1_{*,q}\right).$$
\end{prop}

\begin{proof} : We have ${}^G\!E^1_{p,q} = H_{p+q}\left({}^G\!E^0_{p, *-p}\right)$ and 
$${}^G\!E^0_{p, *-p} = E^{0~(L_*(\mathcal{W}C_*(X)))}_{p,*-p} = L_*\left(E^{0~(\mathcal{W} C_*(X))}_{p,*-p}\right),$$
then ${}^G\!E^1_{p,q} = H_{p+q}\left(G,E^{0}_{p,*-p}\right)$.
\end{proof}

This allows one to consider the two spectral sequences
$${}^{q}_{I}\!E^2_{\alpha,\beta} = H^{-\alpha}\left(G, \widetilde{E}^2_{\beta, q}\right)$$ 
$${}^{~q}_{II}\!E^1_{\alpha,\beta} = H^{-\alpha}\left(G, \widetilde{E}^1_{\beta, q}\right)$$
which both converge to ${}^G\!\widetilde{E}^2_{\alpha + \beta,q} = H_{\alpha + \beta}\left(G, \widetilde{E}^1_{*,q}\right)$.

\begin{rem} The spectral sequence ${}^{~q}_{II}\!E$ depends on the chosen representative of the weight complex with action at the chain level.
\end{rem}

Let $X$ be a real algebraic $G$-variety of dimension $d$. We take a look at the terms of the $d^{\text{th}}$ row of the reindexed equivariant weight spectral sequence :

\begin{cor} \label{equi_weight_spec_max_dim} For all $p \in \mathbb{Z}$,
$$^G\!\widetilde{E}^2_{p,d} = H^{-p}\left(G, \widetilde{E}^2_{0,d}\right)$$
\end{cor}

\begin{proof} We consider the spectral sequence
$${}^{d}_I\!E^2_{\alpha,\beta} = H^{-\alpha}\left(G,\widetilde{E}^2_{\beta, d}\right) \Rightarrow {}^G\!\widetilde{E}^2_{\alpha + \beta,d},$$
We have $\widetilde{E}^2_{\beta,d} = 0$ if $\beta \neq 0$ (\cite{MCP} section 1.3), therefore the spectral sequence collapses at ${}^{d}_{I}\!E^2$ and 
$$^G\!\widetilde{E}^2_{p,d} = \bigoplus_{\alpha + \beta = p} {}^{d}_I\!E^2_{\alpha,\beta} = H^{-p} \left(G, \widetilde{E}^2_{0,d}\right).$$
\end{proof}

In particular, notice that, in the general case, there is an infinity of non-zero terms in the row $q = d$ of the reindexed equivariant spectral sequence. Even if the reindexed equivariant weight spectral sequence is not left bounded, we have the following bounds :

\begin{cor} \label{equiv_weight_spec_bounds} For all $r \geq 2$, $p,q \in \mathbb{Z}$, if $^G\!\widetilde{E}^r_{p,q} \neq 0$ then $0 \leq q \leq d$ and $p+q \leq d$.
\end{cor}

\begin{proof} For all $q \in \mathbb{Z}$, we have the spectral sequence
$${}^{q}_{I}\!E^2_{\alpha,\beta} = H^{-\alpha}\left(G, \widetilde{E}^2_{\beta, q}\right) \Rightarrow {}^G\!\widetilde{E}^2_{\alpha + \beta,q}.$$
If $\alpha > 0$, $H^{-\alpha}(G, \cdot) = 0$ and, according to \cite{MCP} section 1.3., for all $\beta \in \mathbb{Z}$, if $\widetilde{E}^2_{\beta, q} \neq 0$ then $\beta \geq 0$, $q \geq 0$ and $\beta + q \leq d$. We conclude by noticing that if $p+q > d$, for all $\alpha,\beta \in \mathbb{Z}$ such that  $\alpha + \beta = p$, we have $\alpha + \beta + q > d$ and then $\beta + q > d$ or $\alpha > 0$. 
\end{proof}

The equivariant weight filtration is bounded but the smallest non-trivial index depends only on the dimension of the considered variety, which is different from the weight filtration case.

\begin{cor} \label{bounds_weight_fil} The equivariant weight filtration on the equivariant homology of $X$ is a bounded increasing filtration 
$$0 = \Omega_{-d-1} H_k(X;G) \subset \Omega_{-d} H_k(X;G) \subset \ldots \subset \Omega_0 H_k(X;G) = H_k(X;G).$$
\end{cor}

\begin{proof} One can prove $\Omega_0 H_k(X;G) = H_k(X;G)$ and $\Omega_{-d-1} H_k(X;G) = 0$ using the equalities 
$$\Omega_p H_k(X;G) = \bigoplus_{q \geq 0} {}^G\!\widetilde{E}^{\infty}_{k+p-q,-(p-q)}$$
and the previous corollary \ref{equiv_weight_spec_bounds}.
\end{proof}

\begin{rem} The fact that the weight spectral sequence is left bounded is a key tool to extract additive invariants from its level two, namely the virtual Betti numbers (see \cite{MCP} section 1.3). In the equivariant framework, the equivariant weight spectral sequence's long exact sequences of additivity are not finite and furthermore, the condition of compactness-smoothness does not imply the collapsing of the equivariant weight spectral sequence at level two.
\end{rem}

\begin{prop} \label{var_comp-non-sing_equiv} Assume $X$ to be compact and nonsingular. Then the equivariant weight spectral sequence ${}^G\!\widetilde{E}$ of $X$ coincides, from level two, with the Hochschild-Serre spectral sequence 
$${}_{I}\!E^2_{p,q}(X) = H^{-p}\left(G, H_q(X)\right) \Rightarrow H_{p+q}(X ; G)$$ 
associated to $X$ and $G$.
\end{prop}

\begin{proof} If $X$ is compact nonsingular, the weight complex with action ${}^G\!\mathcal{W} C_*(X)$ is quasi-isomorphic to $F^{can} C_*(X)$ in $\mathcal{C}^G$. Since the functor $L$ preserves filtered quasi-isomorphisms, the equivariant weight complex $\Omega C_*^G$ is quasi-isomorphic to $\mathcal{F}^{can} C_*^G = \mathcal{F}^{can} L_*\left(C_*\right)$ in $\mathcal{C}_-$ on compact nonsingular real algebraic $G$-varieties (the same goes for any realization of the equivariant weight complex in $\mathcal{C}^G$).

Then we use lemma \ref{lem_spec_seq_fil_can_equiv} to conclude that, after reindexing, the equivariant weight spectral sequence of $X$ is isomorphic to the Hochschild-Serre spectral sequence of $X$ from level two.
\end{proof}

In particular, even in the case of a compact nonsingular variety, the equivariant weight spectral sequence may not degenerate at level two : see example \ref{ex_sphere_act}.
\\

Below, we compute the equivariant weight spectral sequences and equivariant weight filtrations of a singular real algebraic variety equipped with two different algebraic involutions.

\begin{ex} \label{ex_suite_spec_poids_equiv} Let $X$ be the real algebraic curve given by the equation $y^2 = x^2 - x^4$ in $\mathbb{R}^2$.

\begin{enumerate}
	\item We consider the action of $G = \mathbb{Z}/2\mathbb{Z}$ on $X$ given by the involution $\sigma : (x,y) \mapsto (-x,y)$. The page $\widetilde{E}^2(X)$ of the reindexed weight spectral sequence is given by 
$$\begin{matrix}
\mathbb{Z}_2 [X] & \\
\mathbb{Z}_2 [\{p_0\}] & \mathbb{Z}_ 2[X_1]
\end{matrix}$$
where $p_0 = (0,0)$ is the unique point of $X$ being fixed under the action of $G$, and $X_1$ and $X_2$ are the two $1$-cycles of $X$, exchanged by the action.

We have
\begin{itemize}
	\item ${}^G\!\widetilde{E}^2_{p,1} = H^{-p}\left(G, \widetilde{E}^2_{0,1}\right) = \mathbb{Z}_2[X]$ if $p \leq 0$, and $0$ otherwise.
	\item the terms ${}^G\!\widetilde{E}^2_{p,0}(X)$ are given by the Hochschild-Serre spectral sequence associated to the $G$-complex $\widetilde{E}^1_{*,0}(X)$ : the level two of this spectral sequence is
	$$\begin{matrix}
		\cdots & \mathbb{Z}_2[X_1] & \cdots & \mathbb{Z}_2[X_1] & \mathbb{Z}_2[X_1] \\
		\cdots & \mathbb{Z}_2 [\{p_0\}]  & \cdots & \mathbb{Z}_2 [\{p_0\}] & \mathbb{Z}_2 [\{p_0\}] 
	\end{matrix}$$ 
and since the differentials are trivial, we have 
$${}^G\!\widetilde{E}^2_{p,0}(X) =
\begin{cases}
	\mathbb{Z}_2[X_1] \text{ if $p = 1$,} \\
	\mathbb{Z}_2[X_1]  \oplus \mathbb{Z}_2 [\{p_0\}] \text{ if $p \leq 0$,} \\
	0 \text{ otherwise.}
\end{cases}$$
\end{itemize}

Consequently, the page ${}^G\!\widetilde{E}^2(X)$ is
$$\begin{matrix} 
	\cdots & \mathbb{Z}_2[X] & \cdots & \mathbb{Z}_2[X] & \mathbb{Z}_2[X] & \\
	\cdots & \mathbb{Z}_2[X_1]  \oplus \mathbb{Z}_2 [\{p_0\}]  & \cdots & \mathbb{Z}_2[X_1]  \oplus \mathbb{Z}_2 [\{p_0\}]  & \mathbb{Z}_2[X_1]  \oplus \mathbb{Z}_2 [\{p_0\}] & \mathbb{Z}_2[X_1]
\end{matrix}$$
and the page ${}^G\!\widetilde{E}^3(X)$ is
$$\begin{matrix} 
	\cdots & 0 & \cdots & 0 & \mathbb{Z}_2[X] & \\
	\cdots & \mathbb{Z}_2 [\{p_0\}]  & \cdots & \mathbb{Z}_2 [\{p_0\}]  &  \mathbb{Z}_2 [\{p_0\}] & 0
\end{matrix}$$
because ${}^G\!\widetilde{d}^2([X_1]) = \partial \oplus (1 + \sigma)([X_1]) = [X]$ and ${}^G\!\widetilde{d}^2([\{p_0\}]) = 0$.

As a consequence, the equivariant weight spectral sequence ${}^G\!\widetilde{E}(X)$ collapses at ${}^G\!\widetilde{E}^3(X)$ and the equivariant weight filtration of $X$ with respect to the action of $\sigma$ is given by 
$$\Omega_{-1} H_1(X ; G) = \Omega_0 H_1(X ; G) = \mathbb{Z}_2[X]$$
and
$$0 = \Omega_{-1} H_k(X ; G) \subset \Omega_0 H_k(X ; G) =  \mathbb{Z}_2 [\{p_0\}] $$
for $k \leq 0$.

\item If now we consider the action of $G$ given by $(x,y) \mapsto (x,-y)$, we obtain the same terms for the page ${}^G\!\widetilde{E}^2(X)$ :
$$\begin{matrix} 
	\cdots & \mathbb{Z}_2[X] & \cdots & \mathbb{Z}_2[X] & \mathbb{Z}_2[X] & \\
	\cdots & \mathbb{Z}_2[X_1]  \oplus \mathbb{Z}_2 [\{p_0\}]  & \cdots & \mathbb{Z}_2[X_1]  \oplus \mathbb{Z}_2 [\{p_0\}]  & \mathbb{Z}_2[X_1]  \oplus \mathbb{Z}_2 [\{p_0\}] & \mathbb{Z}_2[X_1]
\end{matrix}$$
However, here, ${}^G\!\widetilde{d}^2([X_1]) = 0$ because the cycle $X_1$ is globally invariant under the action. As a consequence, the differential ${}^G\!\widetilde{d}^2$ is trivial and the equivariant spectral sequence of~$X$ degenerates at ${}^G\!\widetilde{E}^2(X)$. The equivariant weight filtration on the equivariant homology of~$X$ is then given by
$$\mathbb{Z}_2[X] = \Omega_{-1} H_1(X ; G) \subset \Omega_{0} H_1(X ; G) = \mathbb{Z}_2[X_1] \oplus \mathbb{Z}_2[X_2]$$
and
$$\mathbb{Z}_2[X] = \Omega_{-1} H_k(X ; G) \subset \Omega_{0} H_k(X ; G) = \mathbb{Z}_2[X_1] \oplus \mathbb{Z}_2[X_2] \oplus \mathbb{Z}_2 [\{p_0\}]$$
for $k \leq 0$. 
\end{enumerate}
\end{ex}

\subsection{The odd-order group case}

If the order of the group $G$ is odd, the equivariant weight spectral sequence correspond to the elements of the weight spectral sequence being invariant under the action of~$G$, allowing one, in particular, to extract additive invariants. This is because the chains we consider have coefficients in $\mathbb{Z}_2$ : if the order of $G$ is odd, the ring $\mathbb{Z}_2[G]$ is semi-simple according to Maschke theorem (see for example \cite{CTVZ} Theorem 2.1.1).

Consequently, every $\mathbb{Z}_2[G]$-module is projective (and injective). Then  
\begin{equation*}\label{triv_res}... \rightarrow 0 \rightarrow \mathbb{Z}_2 \rightarrow \mathbb{Z}_2 \rightarrow 0 \tag{3.2}\end{equation*}
is a projective resolution of $\mathbb{Z}_2$ over $\mathbb{Z}_2[G]$ and the cohomology of $G$ with coefficients in a $\mathbb{Z}_2[G]$-module $M$ is $H^n\left(G,M\right) = \Hom_{\mathbb{Z}_2[G]}(\mathbb{Z}_2 ,M) = M^G$ for $n=0$, and $0$ otherwise.

Furthermore, the semi-simplicity of $\mathbb{Z}_2[G]$ is equivalent to the condition that every short exact sequence of $\mathbb{Z}_2[G]$-modules is split. In particular, the functor $\Gamma^G$ that associates to any $\mathbb{Z}_2[G]$-module the set of its invariant elements is exact.
\\

Let $G$ be a group of odd order. Choosing the trivial resolution \eqref{triv_res} of $\mathbb{Z}_2$ as a projective resolution over $\mathbb{Z}_2[G]$, in the double complex associated to the functor $L$, the column $p=0$ is the only one which is potentially non-zero and then, for any $G$-complex $K_*$,
$$L_*\left(K_*\right) = \left(K_*\right)^G$$
and
$$H_*\left(G, K_*\right) = H_n\left(L_*\left(K_*\right)\right) = H_n\left((K_*)^G\right) = \left(H_n\left(K_*\right)\right)^G,$$
(the functor $\Gamma^G$ is exact on the category of $\mathbb{Z}_2[G]$-modules).

Finally, the spectral sequences ${}_{I}\!E$ et ${}_{II}\!E$ associated to a $G$-complex $K_*$ coincide and collapse at level two :
$${}_{I}\!E^2_{p,q} = H^{-p}\left(G,H_q\left(K_*\right)\right) = 
\begin{cases} 
\left(H_q(K_*)\right)^G = H_q\left((K_*)^G\right) \mbox{ if } p = 0,
\\ 0 \mbox{ otherwise, }
\end{cases}$$ 
$${}_{II}\!E^2_{p,q} = H_q\left(H^{-p}\left(G,K_*\right)\right) = 
\begin{cases} 
H_q\left((K_*)^G\right) = \left(H_q(K_*)\right)^G \mbox{ if } p = 0,
\\ 0 \mbox{ otherwise. }
\end{cases}$$  
\\

Now, if we consider a real algebraic $G$-variety $X$, we have
$$H_*(X ; G) = H_*\left((C_*(X))^G\right)= \left(H_*(X)\right)^G~,~\Omega C^G_*(X) = (\mathcal{W} C_*(X))^G~,~{}^G\!E(X) = \left(E(X)\right)^G.$$

In particular, the non-zero terms of the reindexed equivariant spectral sequence of $X$ are bounded into the triangle of vertices $(0,0)$, $(0,d)$ and $(d,0)$ (if $d$ is the dimension of $X$). Consequently, we are able to recover the equivariant virtual Betti numbers (see \cite{GF} and section \ref{sect_add_inv} below) from the equivariant weight spectral sequence when the order of the group~$G$ is odd.

\begin{prop} Let $G$ be a finite group of odd order. For all real algebraic $G$-varieties $X$ and all $q \in \mathbb{Z}$, the $q^{\text{th}}$ equivariant virtual Betti number of $X$ (\cite{GF}) is the alternating sum of the dimensions of the terms of the row $q$ of the reindexed equivariant weight spectral sequence :

$$\beta_q^G(X) = \sum_p (-1)^{p \geq 0} \dim_{\mathbb{Z}_2}  {}^G\!\widetilde{E}^2_{p,q}.$$

\end{prop}

\begin{proof} In this set-up, for $q \in \mathbb{Z}$, the additivity long exact sequence for an equivariant closed inclusion (\ref{long_ex_seq_equiv_spec_seq}) is finite and gives us the additivity of the right-hand side expression. Moreover, if $X$ is compact nonsingular, ${}^G\!\widetilde{E}^2_{p,q} = \left(\widetilde{E}^2_{p,q}\right)^G =~0$ if $p \neq 0$ and
$$^G\!\widetilde{E}^2_{0,q} = \left(\widetilde{E}^2_{0,q}\right)^G = \left(H_q(X)\right)^G = H_q(X ; G).$$
\end{proof}

\section{Additivity} \label{sect_add_inv}

In previous section \ref{sect_equiv_weight_fil}, we could not extract finite additive invariants from the equivariant weight spectral sequence in the general case. In this section \ref{sect_add_inv}, we construct bounded double complexes~${}^k\!\widehat{C}$, for all $k \in \mathbb{Z}$, which extend the additivity of the Nash-constructible filtration and of the weight spectral sequence it induces (remark \ref{rem_add_gp_triv} and proposition \ref{suite_ex_long_finie}). For each $k \in \mathbb{Z}$, the columns of one of the spectral sequences induced by ${}^k\!\widehat{C}$ are the columns of spectral sequences ${}^{~q}_{II}\!E$, the ones which converge to the level two of the reindexed equivariant weight spectral sequence (see paragraph \ref{subsec_equi_wei_spec}). We show that these columns provide finite additive invariants in terms of bounded long exact sequences. 

We then consider the other spectral sequence the double complex ${}^k\!\widehat{C}$ induces, focusing on the case of the two-elements group, for which we have the Smith Nash-constructible exact sequence that enables its computation. The abutment of this spectral sequence involves the geometry of invariant chains and the geometry of the fixed points set and each of its columns also yields additive invariants in terms of bounded long exact sequences. Furthermore, we prove that the Euler characteristic of this spectral sequence coincides with the $k^{\text{th}}$ equivariant virtual Betti number in some cases. The equivariant virtual Betti numbers are the unique additive invariants defined on the category of real algebraic $G$-varieties that equal the dimensions of equivariant homology groups on compact nonsingular varieties (see \cite{GF}).

We finally discuss the relation of these additive invariants with equivariant virtual Betti numbers in general case for $G = \mathbb{Z}/2\mathbb{Z}$. The extension theorem \ref{crit_action} can provide a filtered complex inducing a spectral sequence from which we can recover the equivariant virtual Betti numbers. If the invariant semialgebraic chains equipped with the Nash-constructible filtration realized this filtered complex, the spectral sequence induced by the double complex ${}^k\!\widehat{C}$ would compute the equivariant virtual Betti numbers. 

\subsection{The double complex ${}^k\!\widehat{C}$} \label{sec_add}

Let $G$ be a finite group.

Let $X$ be a real algebraic $G$-variety. For the sake of readability, we simply denote $C_*(X)$ by $C_*$. 

The group cohomology of $G$ is functorial and each exact sequence $0 \rightarrow \mathcal{N}_{\alpha-1} C_{\beta} \rightarrow \mathcal{N}_{\alpha} C_{\beta} \rightarrow \frac{\mathcal{N}_{\alpha} C_{\beta}}{\mathcal{N}_{\alpha-1} C_{\beta}} \rightarrow 0$ induces a long exact sequence in group cohomology, compatible with the boundary operator of $C_*$. Therefore, we can construct the following bounded double complex
$$\xymatrix{ &  \ar[d] & \ar[d] & \ar[d] & \\
 & H^{-k-(\alpha-1)}\left(G, \frac{\mathcal{N}_{\alpha-1} C_{\beta+1}}{\mathcal{N}_{\alpha-2} C_{\beta+1}}\right) \ar[d]^{d^1} \ar[l] &  H^{-k-\alpha}\left(G, \frac{\mathcal{N}_{\alpha} C_{\beta+1}}{\mathcal{N}_{\alpha-1} C_{\beta+1}}\right) \ar[d]^{d^1} \ar[l]_{d^0} &  H^{-k-(\alpha+1)}\left(G, \frac{\mathcal{N}_{\alpha+1} C_{\beta+1}}{\mathcal{N}_{\alpha} C_{\beta+1}}\right) \ar[d]^{d^1} \ar[l]_{d^0} & \ar[l] \\
 & H^{-k-(\alpha-1)}\left(G, \frac{\mathcal{N}_{\alpha-1} C_{\beta}}{\mathcal{N}_{\alpha-2} C_{\beta}}\right) \ar[d]^{d^1} \ar[l] &  H^{-k-\alpha}\left(G, \frac{\mathcal{N}_{\alpha} C_{\beta}}{\mathcal{N}_{\alpha-1} C_{\beta}}\right) \ar[d]^{d^1} \ar[l]_{d^0} &  H^{-k-(\alpha+1)}\left(G, \frac{\mathcal{N}_{\alpha+1} C_{\beta}}{\mathcal{N}_{\alpha} C_{\beta}}\right) \ar[d]^{d^1} \ar[l]_{d^0} & \ar[l]\\
 & H^{-k-(\alpha-1)}\left(G, \frac{\mathcal{N}_{\alpha-1} C_{\beta-1}}{\mathcal{N}_{\alpha-2} C_{\beta-1}}\right) \ar[d] \ar[l] &  H^{-k-\alpha}\left(G, \frac{\mathcal{N}_{\alpha} C_{\beta-1}}{\mathcal{N}_{\alpha-1} C_{\beta-1}}\right) \ar[d] \ar[l]_{d^0} &  H^{-k-(\alpha+1)}\left(G, \frac{\mathcal{N}_{\alpha+1} C_{\beta-1}}{\mathcal{N}_{\alpha} C_{\beta-1}}\right) \ar[d] \ar[l]_{d^0} & \ar[l] \\
 & & &  & 
}$$ 
that we denote by $({}^k\!\widehat{C}_{\alpha,\beta}(X))_{(\alpha,\beta) \in \mathbb{Z} \times \mathbb{Z}}$. The differentials $d^1$ are induced by the boundary operator of $C_*$ and the differentials $d^0$ are given by the following commutative diagram
$$\xymatrix{ H^{-k-(\alpha-1)}\left(G, \mathcal{N}_{\alpha-1} C_{\beta}\right) \ar[r] &  H^{-k-(\alpha-1)}\left(G, \frac{\mathcal{N}_{\alpha-1} C_{\beta}}{\mathcal{N}_{\alpha-2} K_{\beta}}\right) \\
H^{-k-\alpha}\left(G, \frac{\mathcal{N}_{\alpha} C_{\beta}}{\mathcal{N}_{\alpha-1} C_{\beta}}\right) \ar[u] \ar[ru]^{d^0} & \\
H^{-k-\alpha}\left(G, \mathcal{N}_{\alpha} C_{\beta}\right)  \ar[u] &  H^{-k-(\alpha+1)}\left(G, \frac{\mathcal{N}_{\alpha+1} C_{\beta}}{\mathcal{N}_{\alpha} C_{\beta}}\right) \ar[l] \ar[lu]^{d^0}
}$$

\begin{rem} \label{rem_add_gp_triv}
For $G = \{e\}$, the double complex $({}^k\!\widehat{C}_{\alpha,\beta})_{(\alpha,\beta) \in \mathbb{Z} \times \mathbb{Z}}$ is reduced to the column $\alpha = - k$ and the homology of this column provides the row $q = k$ of the page two of the reindexed weight spectral sequence of $X$. In particular, the Euler characteristic of the homology of the column $\alpha = - k$ is the $k^{\text{th}}$ virtual Betti number.
\end{rem}

Consider the induced spectral sequence ${}_{II}^{~k}\!\widehat{E}$, computed by first taking the homology of $d^1$ :
$${}_{II}^{~k}\!\widehat{E}^1_{\alpha,\beta} = H_{\beta}\left(H^{-k-\alpha}\left(G,\frac{\mathcal{N}_{\alpha} C_{*}}{\mathcal{N}_{\alpha-1} C_{*}}\right)\right) = {}^{- \alpha}_{~II}\!E^2_{\alpha + k, \alpha + \beta}.$$
Each column of the page ${}_{II}^{~k}\!\widehat{E}^1$ is additive in the following meaning :

\begin{prop} \label{suite_ex_long_finie} Let $Y \subset X$ be an equivariant closed inclusion in $\mathbf{Sch}_c^G(\mathbb{R})$. For any $q$ and~$i$ in $\mathbb{Z}$, there is a finite long exact sequence
$$\cdots \rightarrow {}^{~q}_{II}\!E^2_{i,j}(Y) \rightarrow {}^{~q}_{II}\!E^2_{i,j}(X) \rightarrow {}^{~q}_{II}\!E^2_{i,j}(X \setminus Y) \rightarrow {}^{~q}_{II}\!E^2_{i,j-1}(Y) \rightarrow \cdots,$$
where ${}^{~q}_{II}\!E^2_{i,j} = H_j\left(H^{-i}\left(G, \widetilde{E}^1_{*,q} \right)\right)$, $i,j \in \mathbb{Z}$, is the page two of the spectral sequence ${}^{~q}_{II}\!E$ induced by the Nash-constructible filtration (see paragraph \ref{subsec_equi_wei_spec}).
\end{prop}

\begin{proof} From the exactness of the equivariant short sequences 
\begin{equation*} \label{suite_ex_court_equiv_ch} 0 \rightarrow \mathcal{N}_p C_k(Y) \rightarrow \mathcal{N}_p C_k(X) \rightarrow \mathcal{N}_p C_k(X \setminus Y) \rightarrow 0 \tag{4.1} \end{equation*}
(\cite{MCP} Proof of Theorem 3.6 and \cite{Pri-CA} Remarque 3.9) follows, by a diagram chase, the exactness of the equivariant short sequences 
\begin{equation*} \label{suite_ex_court_suite_spec} 0 \rightarrow E^0_{p,q}(Y) \rightarrow E^0_{p,q}(X) \rightarrow E^0_{p,q}(X \setminus Y) \rightarrow 0 \tag{4.2} \end{equation*}
which then induce the long exact sequences in group cohomology  
$$\cdots \rightarrow H^{k}\left(G,E^0_{p,q}(Y)\right) \rightarrow H^{k}\left(G,E^0_{p,q}(X)\right) \rightarrow H^{k}\left(G,E^0_{p,q}(X \setminus Y)\right) \rightarrow H^{k+1}\left(G,E^0_{p,q}(Y)\right) \cdots$$
Since, for any $k$ and $p$ in $\mathbb{Z}$, the short exact sequence (\ref{suite_ex_court_equiv_ch}) is split by the closure morphism $c \in \mathcal{N}_p C_k(X \setminus Y) \mapsto \overline{c} \in \mathcal{N}_p C_k(X)$, which is an equivariant morphism by \cite{Pri-CA} Proposition 2.4, so is, for any $p,q \in \mathbb{Z}$, the short exact sequence (\ref{suite_ex_court_suite_spec}) by the induced equivariant morphism $E^0_{p,q}(X \setminus Y) \rightarrow E^0_{p,q}(X)$, and the short sequence of complexes
$$0 \rightarrow H^{k}\left(G,E^0_{p,*}(Y)\right) \rightarrow H^{k}\left(G,E^0_{p,*}(X)\right) \rightarrow H^{k}\left(G,E^0_{p,*}(X \setminus Y)\right) \rightarrow 0$$
is exact (by functoriality of $H^k\left(G, \cdot\right)$). We obtain the result by considering the induced long exact sequences in homology.
\end{proof}

If all the spaces ${}_{II}^{~k}\!\widehat{E}^1$ (or equivalently all the spaces ${}^{~q}_{II}\!E^2_{i,j}$ such that $q + i =k$) are finite-dimensional, we can define the quantity 
$$B^G_k(X) := \chi\left({}_{II}^{~k}\!\widehat{E}^1\right) = \sum_{q+i = k} \sum_{j \geq 0} (-1)^j \dim_{\mathbb{Z}_2} {}^{~q}_{II}\!E^2_{i,j}(X)$$
and, as a consequence of the previous proposition \ref{suite_ex_long_finie}, we have 
\begin{theo} \label{inv_add_B_th} The invariants $B_k^G( \cdot)$ are additive on real algebraic $G$-varieties on which they are well-defined.
\end{theo}

\begin{rem} The additivity of the spectral sequences ${}^{~q}_{II}\!E$ and of the invariants $B^G_k$ strongly depends on the choice of the Nash-constructible filtration to realize the weight complex with action.
\end{rem}

\begin{ex} 
\begin{enumerate}
	\item Let $X$ be the real algebraic curve defined by the equation $y^2 = x^2 - x^4$ in $\mathbb{R}^2$ and consider the $\mathbb{Z}/2\mathbb{Z}$-action given by $(x,y) \mapsto (-x,y)$. Keeping the notations of example \ref{ex_suite_spec_poids_equiv}, we have
$${}^{~q}_{II}\!E^2_{i,j}(X) =
\begin{cases} 
\mathbb{Z}_2[X] \text{ if $q = 1$, $i \leq 0$ and $j = 0$,} \\
\mathbb{Z}_2 [X_1] \text{ if $q = 0$, $i \leq 0$ and $j = 1$,} \\
\mathbb{Z}_2 [\{p_0\}] \text{ if $q = 0$, $i \leq 0$ and $j = 0$,} \\
0 \text{ otherwise,}
\end{cases} 
\mbox{~~~~~~and~~~~~~}
B_k^G(X) =
\begin{cases}
1 \text{ if $k \leq 1$,} \\
0 \text{ otherwise.}
\end{cases}$$

	\item Take the same variety $X$ and equip it with the action given by $(x,y) \mapsto (x,-y)$. We have 
$${}^{~q}_{II}\!E^2_{i,j}(X) =
\begin{cases} 
\mathbb{Z}_2[X] \text{ if $q = 1$, $i \leq 0$ and $j = 0$,} \\
\mathbb{Z}_2 [X_1] \text{ if $q = 0$, $i = 0$ and $j = 1$,} \\
\mathbb{Z}_2 [\{p_0\}] \oplus \mathbb{Z}_2 [\{p_1\}] \text{ if $q = 0$, $i \leq 0$ and $j = 0$,} \\
0 \text{ otherwise,}
\end{cases}
\mbox{and~~}
B_k^G(X) =
\begin{cases}
1 \text{ if $k = 1$,} \\
2 \text{ if $k = 0$,} \\
3 \text{ if $k < 0$,} \\
0 \text{ otherwise,}
\end{cases}$$
(where $p_1$ is one of the two points being invariant under the action besides $p_0$).

\end{enumerate}
\end{ex}

\begin{rem} In each case and for all $k \in \mathbb{Z}$, $B^G_k(X)$ is equal to the $k^{\text{th}}$ equivariant virtual Betti number $\beta^G_k(X)$ of $X$ (see \cite{GF} Example 4.6).
\end{rem}

In the next paragraph, we study the other spectral sequence induced by the double complex~${}^k\!\widehat{C}$, focusing on the case $G = \mathbb{Z}/2\mathbb{Z}$. We use the Smith Nash-constructible exact sequence (theorem \ref{nash-smith_ex}) to compute it and extract additivity from this spectral sequence as well.

\subsection{The case $G = \mathbb{Z}/2\mathbb{Z}$} \label{spec_inv_z2z}

We fix $G := \mathbb{Z}/2\mathbb{Z}$.

Let $X$ be a real algebraic $G$-variety and let $k$ be an integer. We compute the spectral sequence ${}_{I}^{k}\!\widehat{E}(X)$ induced by the double complex $({}^k\!\widehat{C}_{\alpha,\beta})_{(\alpha,\beta) \in \mathbb{Z} \times \mathbb{Z}}$ by first taking the homology of $d^0$. We use the Smith Nash-constructible short exact sequence (theorem \ref{nash-smith_ex}) to show that
$${}^k_I\!\widehat{E}^1_{\alpha, \beta} =
\begin{cases}
\frac{\mathcal{N}_{\alpha} C_{\beta}\left(X^G\right)}{\mathcal{N}_{\alpha-1} C_{\beta}\left(X^G\right)} \mbox{ if } - k - \alpha \geq 1, \\
\\
 \frac{(\mathcal{N}_{\alpha} C_{\beta})^G}{(\mathcal{N}_{\alpha-1} C_{\beta})^G} \mbox{  if } - k - \alpha = 0,\\
 \\
 0 \mbox{ otherwise, }
\end{cases}$$
and consequently, the page ${}^{k}_{I}\!\widehat{E}^2$ is given by :
$$\xymatrix{ 
\cdots & H_{\beta + 1}\left(\frac{\mathcal{N}_{\alpha} C_{*}\left(X^G\right)}{\mathcal{N}_{\alpha - 1} C_{*}\left(X^G\right)}\right) & \cdots & H_{\beta + 1} \left(\frac{\mathcal{N}_{-k-1} C_{*}\left(X^G\right)}{\mathcal{N}_{-k-2} C_{*}\left(X^G\right)} \right) & H_{\beta+1}\left(\frac{(\mathcal{N}_{-k} C_{*})^G}{(\mathcal{N}_{-k-1} C_{*})^G}\right) \\
\cdots & H_{\beta}\left(\frac{\mathcal{N}_{\alpha} C_{*}\left(X^G\right)}{\mathcal{N}_{\alpha - 1} C_{*}\left(X^G\right)}\right) &  \cdots  & H_{\beta}\left(\frac{\mathcal{N}_{-k-1} C_{*}\left(X^G\right)}{\mathcal{N}_{-k-2} C_{*}\left(X^G\right)}\right) & H_{\beta}\left(\frac{(\mathcal{N}_{-k} C_{*})^G}{(\mathcal{N}_{-k-1} C_{*})^G}\right) \\
\cdots & H_{\beta-1}\left(\frac{\mathcal{N}_{\alpha} C_{\beta-1}\left(X^G\right)}{\mathcal{N}_{\alpha - 1} C_{\beta-1}\left(X^G\right)}\right) &  \cdots &  H_{\beta - 1}\left(\frac{\mathcal{N}_{-k-1} C_{*}\left(X^G\right)}{\mathcal{N}_{-k-2} C_{*}\left(X^G\right)} \right)&     H_{\beta-1} \left(\frac{(\mathcal{N}_{-k} C_{*})^G}{(\mathcal{N}_{-k-1} C_{*})^G} \right)
}$$
(the first non-zero column from the right side is the column $\alpha = -k$). In particular, the spectral sequence $\left({}^{k}_{I}\!\widehat{E}, {}^{k}\!\delta\right)$ degenerates at level two and one can compute the homology of the total complex associated to the double complex ${}^k\!\widehat{C}$. 

Indeed, let $\omega$ be an element of ${}^{k}_{I}\!\widehat{E}^2_{\alpha,\beta} = H_{\beta}\left(\frac{\mathcal{N}_{\alpha} C_{*}\left(X^G\right)}{\mathcal{N}_{\alpha - 1} C_{*}\left(X^G\right)}\right)$ with $\alpha < -k$. It can be represented by a chain $c \in \mathcal{N}_{\alpha} C_{\beta}(X^G)$ such that $\partial c \in \mathcal{N}_{\alpha-1} C_{\beta-1}(X^G)$. In order to compute the image of the element $\omega$ by ${}^{k}\!\delta^2$, we first apply the differential $d^1$ to its representative $\overline{c}$ in ${}^k\!\widehat{C}_{\alpha, \beta} = H^{-k-\alpha}\left(G, \frac{\mathcal{N}_{\alpha}C_{\beta}}{\mathcal{N}_{\alpha-1}C_{\beta}}\right)$, where $\overline{c}$ denotes the class of $c$ in $H^{-k-\alpha}\left(G, \frac{\mathcal{N}_{\alpha}C_{\beta}}{\mathcal{N}_{\alpha-1}C_{\beta}}\right)$. Then $d^1(\overline{c}) = \overline{\partial c} = 0 \in {}^k\!\widehat{C}_{\alpha, \beta-1} = H^{-k-\alpha}\left(G,\frac{\mathcal{N}_{\alpha}C_{\beta-1}}{\mathcal{N}_{\alpha-1}C_{\beta-1}}\right)$, since $\partial c \in \mathcal{N}_{\alpha-1} C_{\beta-1}(X^G) \subset  \mathcal{N}_{\alpha-1} C_{\beta-1}(X)$. Therefore, the zero element $c' = 0$ of ${}^k\!\widehat{C}_{\alpha+1, \beta-1}$ is such that $d^1(\overline{c}) = d^0(c')$, and finally $d^1(c') = 0 \in {}^k\!\widehat{C}_{\alpha+1, \beta-2}$ is a representative of ${}^{k}\!\delta^2(\omega) \in {}^{k}_{I}\!\widehat{E}^2_{\alpha+1,\beta-2}$. As a result, ${}^{k}\!\delta^2(\omega) = 0$. 

Furthermore, we read additive invariants on each column of the page ${}^{k}_{I}\!\widehat{E}^2$ : the columns $\alpha < -k$ are rows of the page $\widetilde{E}^2(X^G)$ of the reindexed weight spectral sequence of $X^G$ and the split short exact sequences (\ref{suite_ex_court_equiv_ch}) also induce long exact sequences of homology of the pairs $\left( (\mathcal{N}_p C_*)^G,  (\mathcal{N}_p C_*)^G \right)$.
\\

When all the vector spaces $H_{\beta}\left(\frac{(\mathcal{N}_{-k} C_{*})^G}{(\mathcal{N}_{-k-1} C_{*})^G}\right)$ are finite-dimensional, we can consider the Euler characteristic of the page ${}^{k}_{I}\!\widehat{E}^2$ and this defines an additive invariant on the category of real algebraic $G$-varieties on which it is well-defined. When the involved quantities are well-defined, we have the following formula :

\begin{prop} \label{form_b-k-g} For $G = \mathbb{Z}/2\mathbb{Z}$, for all $k \in \mathbb{Z}$ and all real algebraic $G$-varieties $X$ for which the following quantities are well-defined, we have
$$B_k^G(X) =  (-1)^{k} \chi\left(H_{*}\left(\frac{(\mathcal{N}_{-k} C_{*})^G}{(\mathcal{N}_{-k-1} C_{*})^G}\right) \right) + \sum_{q \geq k + 1} \beta_q\left(X^G\right).$$
\end{prop}

\begin{proof} When all the spaces $H_{\beta}\left(\frac{(\mathcal{N}_{-k} C_{*})^G}{(\mathcal{N}_{-k-1} C_{*})^G}\right)$ are finite-dimensional, we have 
$$\chi\left({}^{k}_{I}\!\widehat{E}^2\right) = (-1)^{k} \chi\left(H_{*}\left(\frac{(\mathcal{N}_{-k} C_{*})^G}{(\mathcal{N}_{-k-1} C_{*})^G}\right) \right) + \sum_{\alpha \leq -k-1} (-1)^{\alpha} \chi\left(\widetilde{E}^2_{\alpha + \beta, -\alpha}\left(X^G\right)\right)$$ 
and, for all $\alpha \in \mathbb{Z}$, $\chi\left( \widetilde{E}^2_{\alpha + *, -\alpha}\left(X^G\right)\right) = (-1)^{\alpha} \beta_{-\alpha} \left(X^G\right)$ by \cite{MCP} section 1.3. We conclude by recalling that the Euler characteristics of two spectral sequences converging to the same homology are equal from the level where they are well-defined (see for example \cite{McCleary}). 
\end{proof}

We list below cases where the quantity $\chi({}^{k}_{I}\!\widehat{E}^2)$ is well-defined and coincides with the $k^{\text{th}}$ equivariant virtual Betti number of $X$ (\cite{GF}).

\begin{theo} \label{coinc_b-k-g_eq-virt-bet} Let $G = \mathbb{Z}/2\mathbb{Z}$ and $X$ be a real algebraic $G$-variety. We have 
$$\chi({}^{k}_{I}\!\widehat{E}^2) = \beta_k^G(X)$$
if \begin{itemize}
	\item $k < 0$ and $X$ is any real algebraic $G$-variety,
	\item $X$ is a compact real algebraic $G$-variety equipped with a fixed-point free action, for all $k \in \mathbb{Z}$,
	\item $\dim X = d$ and $k = d$ ($X$ is any real algebraic $G$-variety),
	\item $X$ is a $1$-dimensional real algebraic $G$-variety, for all $k \in \mathbb{Z}$.
\end{itemize}
\end{theo}

\begin{proof}

\begin{enumerate} 
\item If $k < 0$, we have $\chi({}^{k}_{I}\!\widehat{E}^2) = \sum_{q \geq k + 1} \beta_q\left(X^G\right) = \sum_{q \geq 0} \beta_q\left(X^G\right)$. If now we assume $X$ to be compact nonsingular, then the fixed points set $X^G$ is also compact nonsingular and $\chi({}^{k}_{I}\!\widehat{E}^2) = \sum_{q\geq 0} \dim_{\mathbb{Z}_2} H_q\left(X^G\right) = \dim_{\mathbb{Z}_2} H_k(X ; G),$
since $k<0$, by lemma \ref{lem_formula_hom_equiv}.

\item To show the second point, suppose $X$ is a compact real algebraic $G$-variety equipped with a fixed-point free action and let $k \in \mathbb{Z}$. Then $X_{\mathbb{R}}/G$ is an arc-symmetric set and, by proposition \ref{fil_Nash_quotient}, there is an isomorphism of complexes $(\mathcal{N}_{-k} C_{*}(X))^G \cong \mathcal{N}_{-k} C_{*}(X_{\mathbb{R}}/G)$. Then all homology groups $H_{\beta}\left(\frac{(\mathcal{N}_{-k} C_{*})^G}{(\mathcal{N}_{-k-1} C_{*})^G}\right) = H_{\beta}\left(\frac{\mathcal{N}_{-k} C_{*}(X_{\mathbb{R}}/G)}{\mathcal{N}_{-k-1} C_{*}(X_{\mathbb{R}}/G)}\right) = \widetilde{E}^2_{-k+\beta,k}(X_{\mathbb{R}}/G)$ are finite-dimensional and, since $X^G = \emptyset$, we have 
$$\chi({}^{k}_{I}\!\widehat{E}^2) = (-1)^k \sum_{\beta} (-1)^{\beta} \dim_{\mathbb{Z}_2} \widetilde{E}^2_{-k+\beta,k}(X_{\mathbb{R}}/G) = \beta_k(X_{\mathbb{R}}/G).$$
The formula $\beta_k(X_{\mathbb{R}}/G) = \beta_k^G(X)$ of \cite{GF} Proposition 3.15 provides the result.

\item If $d$ is the dimension of $X$, we have $H_{n}\left(\frac{(\mathcal{N}_{-d} C_{*}(X))^G}{(\mathcal{N}_{-d-1} C_{*}(X))^G}\right) = 
\begin{cases} 
(\mathcal{N}_{-d} C_{d}(X))^G \mbox{ if $n = d$,}
\\ 0 \mbox{ otherwise.}
\end{cases}$
Then $\chi({}^{-d}_{I}\!\widehat{E}^2) = \dim_{\mathbb{Z}_2} (\mathcal{N}_{-d} C_{d}(X))^G = \dim_{\mathbb{Z}_2} (\widetilde{E}^2_{0,d}(X))^G$ (recall that $\mathcal{N}_{\alpha} C_{\beta}(X) = 0$ for $\alpha < -d$) and, if $X$ is compact nonsingular, $(\mathcal{N}_{-d} C_{d}(X))^G = (ker~\partial_d)^G = H_d\left((C_*(X))^G \right) = H_d(X ; G)$ (see lemma \ref{lem_formula_hom_equiv}).

\item Finally assume $X$ to be a $1$-dimensional real algebraic $G$-variety. Then
$$H_{n}\left(\frac{( C_{*}(X))^G}{(\mathcal{N}_{-1} C_{1}(X))^G}\right) =
\begin{cases}
\frac{\left(ker~\partial_1\right)^G}{(\mathcal{N}_{-1} C_{*}(X))^G} \mbox{ if $n = 1$,}
\\ H_0\left((C_*(X))^G\right) \mbox{ if $n = 0$,}
\\ 0 \mbox{ otherwise,}
\end{cases}
$$
and $\chi({}^{0}_{I}\!\widehat{E}^2)$ is well-defined ($H_0\left((C_*(X))^G\right)$ is a subspace of $H_0(X ; G)$ by lemma \ref{lem_formula_hom_equiv}). Furthermore, if $X$ is compact nonsingular, we have $(\mathcal{N}_{-1} C_{1}(X))^G = (ker~\partial_1)^G$ and $\chi({}^{0}_{I}\!\widehat{E}^2) = \dim_{\mathbb{Z}_2} H_0\left((C_*(X))^G\right) + \dim_{\mathbb{Z}_2} H_1\left(X^G\right) = \dim_{\mathbb{Z}_2} H_0(X ; G)$.
\end{enumerate}
\end{proof}

Below, we use the extension criterion \ref{crit_action} to show the existence of a filtered complex and a spectral sequence from which we can recover the equivariant virtual Betti numbers. This will allow us to establish sufficient conditions for the Euler characteristics $\chi({}^{k}_{I}\!\widehat{E}^2)$ to be well-defined and to coincide with the equivariant virtual Betti numbers in general case.

\begin{prop} \label{compl_poids_inv} The functor $\mathbf{V}^G(\mathbb{R}) \longrightarrow H o \, \mathcal{C}$ which associates to every projective nonsingular $G$-variety the filtered complex $\left(F^{can} C_*(X)\right)^G = F^{can} (C_*(X))^G$ admits an extension to a functor 
$${}_G\!\mathcal{W} C_{*} : \mathbf{Sch}_c^G(\mathbb{R}) \longrightarrow H o \, \mathcal{C}$$
which verifies acyclicity and additivity properties. Furthermore, ${}_G\!\mathcal{W} C_{*}$ is unique up to a unique isomorphism of $H o \, \mathcal{C}$ with these extension, acyclicity and additivity properties.
\end{prop} 

\begin{proof} We use the extension theorem \ref{crit_action}. Since the functor $\mathbf{V}^G(\mathbb{R}) \longrightarrow H o \, \mathcal{C}~;~X \mapsto \left(F^{can} C_*(X)\right)^G$ factorizes through $\mathcal{C}$, it is $\Phi$-rectified and, if $X$ and $Y$ are two real algebraic $G$-varieties, we have $(F^{can} C_*(X \sqcup Y))^G = (F^{can} C_*(X))^G  \oplus (F^{can} C_*(Y))^G$ therefore it also verifies the condition (F1) of theorem \ref{crit_action}.

The verification of condition (F2) runs as in \cite{MCP}, proof of Theorem 1.1, replacing the homology of the chains by the homology of the invariant chains. Indeed, the short sequences for an equivariant blowing-up
$$0 \rightarrow H_k((C_*(\widetilde{Y}))^G) \rightarrow H_k((C_*(Y))^G) \oplus H_k((C_*(\widetilde{X}))^G) \rightarrow H_k((C_*(X))^G) \rightarrow 0$$
are exact (by lemma \ref{lem_formula_hom_equiv}, for every $G$-variety $Z$, $H_k\left((C_*(Z))^G\right) =  H_k(Z ; G) / \bigoplus_{i \geq k +1} H_i(Z^G)$ and the short sequences for $H_*(\,\cdot\,)$ and $H(\, \cdot \, ; G)$ are exact : see \cite{MCP-VB} proof of Proposition 2.1 and \cite{GF} lemma 3.6).
\end{proof}

If $X$ is a real algebraic $G$-variety, we have $H_*({}_G\!\mathcal{W} C_*(X)) = H_*\left((C_*(X))^G\right)$ (because the short exact sequences of semialgebraic chains for an equivariant closed inclusion and an acyclic square in $\mathbf{Sch}_c^G(\mathbb{R})$ are split by equivariant morphisms) and we can read the acyclicity and additivity properties of ${}_G\!\mathcal{W} C_*$ on the induced spectral sequence, denoted by ${}_G\!E$. We recover the equivariant virtual Betti numbers from the spectral sequence ${}_G\!E$ together with the virtual Betti numbers of the fixed points set :

\begin{theo} \label{nbres_betti_virt_equiv_inv_compl_inv} For every real algebraic $G$-variety $X$, if we reindex the spectral sequence ${}_G\!E$ as in section \ref{subsec_equi_wei_spec}, we have
$$\beta^G_q(X) =  \sum_{p \geq 0} (-1)^p \dim {}_G\!\widetilde{E}^2_{p,q}(X) + \sum_{i \geq q + 1} \beta_i\left(X^G\right)$$
for all $q \in \mathbb{Z}$.
\end{theo}

\begin{proof} First, if $X$ is a $G$-variety of dimension $d$, the terms ${}_G\!\widetilde{E}^2_{p,q}(X)$ are finite-dimensional and bounded in a triangle with vertices $(0,0)$, $(0,d)$ and $(d,0)$. This comes from the isomorphism between the spectral sequence ${}_G\!\widetilde{E}(X)$ of a compact $G$-variety $X$ and the spectral sequence associated to an equivariant cubical hyperresolution of $X$, which is analog to the one considered in the section 1.3 of \cite{MCP} (see Proposition 1.8 and Corollary 1.10 of \cite{MCP}), replacing the semialgebraic chains by the invariant semialgebraic chains.  

Now fix $q \in \mathbb{Z}$. The additivity of the right-hand side expression above is deduced from the additivity property of ${}_G\!\mathcal{W} C_*$, as in \cite{MCP} section 1.3, and the additivity of virtual Betti numbers. If now we consider a compact nonsingular $G$-variety $X$, the filtered complex ${}_G\!\mathcal{W} C_*(X)$ is quasi-isomorphic in $\mathcal{C}$ to $(F^{can} C_*(X))^G$ (the inclusion of categories $\mathbf{V}^G(\mathbb{R}) \rightarrow \mathbf{Reg}^G_{comp}(\mathbb{R})$ has the extension property and the functor $X \mapsto \left(F^{can} C_*(X)\right)^G$ is additive and acyclic in $\mathbf{Reg}^G_{comp}(\mathbb{R})$) and
\begin{eqnarray*}
\sum_{p \geq 0} (-1)^p \dim {}_G\!\widetilde{E}^2_{p,q}(X) + \sum_{i \geq q + 1} \beta_i\left(X^G\right) & = & \dim_{\mathbb{Z}_2} H_q\left((C_*(X))^G\right) + \sum_{i \geq q+1} \dim_{\mathbb{Z}_2} H_i\left(X^G\right) \\
& = & \dim_{\mathbb{Z}_2} H_q(X ; G)
\end{eqnarray*}
by lemma \ref{lem_formula_hom_equiv}.
\end{proof}

Consequently, if the Nash-constructible filtration realized the functor ${}_G\!\mathcal{W} C_*$, we would have, for all real algebraic $G$-varieties and all $k \in \mathbb{Z}$,
\begin{eqnarray*}
\chi({}^{k}_{I}\!\widehat{E}^2) & = & (-1)^{k} \chi\left(H_{*}\left(\frac{(\mathcal{N}_{-k} C_{*})^G}{(\mathcal{N}_{-k-1} C_{*})^G}\right) \right) + \sum_{i \geq k + 1} \beta_i\left(X^G\right) \\
& = & \sum_{p \geq 0} (-1)^p \dim {}_G\!\widetilde{E}^2_{p,k}(X) + \sum_{i \geq k + 1} \beta_i\left(X^G\right) \\
& = & \beta^G_k(X)
\end{eqnarray*}

We already know that the functor $\mathbf{Sch}_c^G(\mathbb{R}) \longrightarrow H o \, \mathcal{C}~;~X \mapsto \left(\mathcal{N} C_*(X)\right)^G$ verifies acyclicity and additivity properties at the chain level thanks to the equivariant splittings of the short exact sequences of additivity (\ref{suite_ex_court_equiv_ch}) and of the analog short exact sequences of acyclicity (this is true for any finite group $G$). It remains to decide whether the equivariant filtered quasi-isomorphism $\mathcal{N} C_*(X) \rightarrow F^{can} C_*(X)$ for $X$ a compact and nonsingular real algebraic $G$-variety (see \cite{MCP}) is preversed by the functor $\Gamma^G$ or not. The clear identification of the behaviour of invariant chains under this morphism should require precise techniques of equivariant Nash geometry.

If the double complex ${}^k\!\widehat{C}$ does not realize the equivariant virtual Betti numbers, this means we identified new additive invariants on real algebraic $G$-varieties in terms of finite long exact sequences. The spectral sequences induced by the double complexes ${}^k\!\widehat{C}$, which are natural invariants of the equivariant geometry of real algebraic $G$-varieties, should be paired with the equivariant weight spectral sequence to which they are related.

 \vspace{0.5cm}
Fabien PRIZIAC
\\
Universit\'e de Bretagne Occidentale
\\
Laboratoire de Math\'ematiques de Bretagne Atlantique
\\
6, avenue Le Gorgeu CS 93837
\\
29238 BREST Cedex 3 (France)
\\
priziac.fabien@gmail.com

\end{document}